\documentclass[11pt]{article}
\usepackage[letterpaper]{geometry}
\usepackage{amsthm,amsmath,amsfonts}
\usepackage{color,graphicx}
\usepackage[colorlinks]{hyperref}
\usepackage[numbers,sort]{natbib}
\graphicspath{{../figures/}}
\usepackage{enumerate}
\usepackage{multirow}
\usepackage{cases}
\newcommand{\R}{\mathbb{R}}

\newcommand{\E}{\mathbb{E}}

\providecommand{\abs}[1]{\lvert#1\rvert}
\providecommand{\norm}[1]{\lVert#1\rVert}
\providecommand{\diag}{{\rm diag}}
\providecommand{\argmax}{{\rm argmax}}

\newcommand{\rs}[1]{{\mbox{\scriptsize \sc #1}}}

\newcommand{\sr}[1]{{\cal #1}}
\newcommand{\dd}[1]{\mathbb{#1}}
\newcommand{\ray}{{\rm r}}

\newcommand{\br}[1]{\left\langle #1 \right\rangle}
\newcommand{\brb}[1]{\big\langle #1 \big\rangle}

\newtheorem{lemma}{Lemma}
\newtheorem{example}{Example}
\newtheorem{proposition}{Proposition}
\newtheorem{corollary}{Corollary}
\newtheorem{theorem}{Theorem}
\theoremstyle{definition}
\numberwithin{equation}{section}

\newtheorem{remark}{Remark}
\newcommand{\eq}[1]{(\ref{eq:#1})}

\newcommand{\lem}[1]{Lemma~\ref{lem:#1}}
\newcommand{\cor}[1]{Corollary~\ref{cor:#1}}
\newcommand{\thr}[1]{Theorem~\ref{thr:#1}}

\newcommand{\pro}[1]{Proposition~\ref{pro:#1}}

\newcommand{\fig}[1]{Figure~\ref{fig:#1}}

\newcommand{\sectn}[1]{Section~\ref{sect:#1}}

\newcommand{\sect}[1]{\ref{sect:#1}}

\begin{document} 

\title{\Large A multi-dimensional SRBM: Geometric views of its product
  form stationary distribution\footnote{Research supported in part by NSF Grants 
CMMI-1030589, CNS-1248117, CMMI-1335724, and JSPS Grant 24310115}}

\author{
J. G. Dai\footnote{School of Operations Research and Information
  Engineering, Cornell University, Ithaca, NY 14853; on leave from
  Georgia Institute of Technology,
jim.dai@cornell.edu}, \quad 
Masakiyo Miyazawa\footnote{Department of Information Sciences, Tokyo
    University of Science, Noda, Chiba 278-0017, Japan;
    miyazawa@is.noda.tus.ac.jp},
{ and }
Jian Wu\footnote{School of Operations Research and Information Engineering, Cornell University, Ithaca, NY 14853, jw926@cornell.edu
} 
} 
\date{May 8, 2014}

\maketitle

\begin{abstract} 
  We present a geometric interpretation of a product form stationary
  distribution for a $d$-dimensional semimartingale reflecting
  Brownian motion (SRBM) that lives in the nonnegative orthant. The
  $d$-dimensional SRBM data can be equivalently specified by $d+1$
  geometric objects: an ellipse and $d$ rays. Using these geometric
  objects, we establish necessary and sufficient conditions for
  characterizing product form stationary distribution. The key idea in
  the characterization is that we decompose the $d$-dimensional
  problem to $\frac{1}{2}d(d-1)$ two-dimensional SRBMs, each of which
  is determined by an ellipse and two rays. This characterization
  contrasts with the algebraic condition of \citet{HarrWill1987a}. A
  $d$-station tandem queue example is presented to illustrate how the
  product form can be obtained using our characterization.  Drawing
  the two-dimensional results in \cite{AvraDaiHase2001,DaiMiya2013},
  we discuss potential optimal paths for a variational problem
  associated with the three-station tandem queue. Except Appendix
  \ref{sec:equiv-two-vers}, the rest of this paper is almost identical to the
  QUESTA paper with the same title. 
\end{abstract}

\section{Introduction}
\label{sect:introduction}
A multidimensional semimartingale reflecting Brownian motion (SRBM)
has been extensively studied in the past as it serves as the diffusion
approximation of a multiclass queueing network and even a more general
stochastic network; see, e.g., \cite{HarrWill1987,
    HarrNguy1993}. In this paper, we focus on a $d$-dimensional 
SRBM $Z=\{Z(t); t \ge 0\}$ that lives on the nonnegative orthant
$\mathbb{R}^d_+$. Its data consists of a (nondegenerate) $d\times
d$ covariance matrix $\Sigma$, a drift vector $\mu\in \R^d$ and a
$d\times d$ reflection matrix $R$.  
An SRBM $Z$ associated with data $(\Sigma, \mu, R )$ is
  defined as a (weak) solution of the 
  following equations: 
\begin{eqnarray}
&&  Z(t) = Z(0) +  X(t) + RY(t) \in \dd{R}^d_+, \quad t\ge 0, \label{eq:RBM1}\\
&&  X=\{X(t), t\ge 0\} \text{ is a  $(\Sigma, \mu)$-Brownian motion},\qquad\label{eq:RBM2} \\
&&  Y(0)=0, Y(\cdot) \text{ is nondecreasing,}\label{eq:RBM3}\\
&&  \int_0^\infty Z_i(t) d Y_i(t) =0 \text{ for } i=1, \ldots, d. \label{eq:RBM4}
\end{eqnarray}
see, e.g., Definition 1 of \cite{DaiHarr1992} for a precise definition.
Thus, in the interior of the orthant, $Z$ behaves as an ordinary Brownian motion with drift vector $\mu $ and covariance matrix $\Sigma $, and $Z$ is pushed in direction  $R^{(j)}$ whenever $Z$ hits the boundary surface $\{z\in \R^d_+: z_{j}=0\}$, where $R^{(j)}$ is the $j$th column of $R$, for $j = 1, {\ldots}, d$. 

A square matrix $A$ is said to be an $\mathcal{S}$-\textit{matrix} if
there exists a vector $w \ge $ 0 such that $A w>0$. (Hereafter, we use
inequalities for vectors as componentwise inequalities.) It is known
that $Z$ exists and is unique in law for each initial
  distribution of $Z(0)$ if and only if $R$ is a completely-$\mathcal{S}$ matrix, that
is, if every principal submatrix of $R$ is an $\sr{S}$ matrix (see,
e.g., \cite{ReimWill1988,DaiWill1995}). We refer to the solution 
$Z$ as $(\Sigma, \mu, R)$-SRBM if the data $\Sigma$, $\mu$ and $R$
need to be specified.

In this paper, we are also concerned with $R$ being a $\sr{P}$ matrix,
which is a square matrix whose principal minors are positive, that is,
each principal sub-matrix has a positive determinant. A
$\mathcal{P}$-matrix is within a subclass of completely-$\mathcal{S}$
matrices; the still more restrictive class of $\mathcal{M}$-matrices
is defined as in Chapter 6 of \cite{BermPlem1979}. 

It is also known that the existence of a stationary distribution for $Z$ requires
\begin{eqnarray}
\label{eq:stability}
  \mbox{$R$ is nonsingular, and $R^{-1} \mu < 0$,}
\end{eqnarray}
but this condition is generally not sufficient (see, e.g.,\cite{BramDaiHarr2010}). 

For applications of the $d$-dimensional SRBM, it is important to obtain
the stationary distribution in a tractable form. However, this is a
very hard problem even for $d=2$.
\citet{HarrWill1987a} show that the $d$-dimensional SRBM  has a
product form  stationary distribution if and only if
 the following
skew symmetry condition 
\begin{eqnarray}
  \label{eq:skew symmetric}
   2\Sigma =R\diag(R)^{-1}\diag(\Sigma) + \diag(\Sigma)\diag(R)^{-1}R^{\rs{t}}
\end{eqnarray}
is satisfied.
Here, for a matrix $A$, $\diag(A)$
denotes the diagonal matrix whose entries are diagonals of $A$, and
$A^T$ denotes the transpose of $A$.  Although many SRBMs arising from
queueing networks do not have product form stationary distributions,
approximations based on product form have been developed to assess the
performance of queueing networks; see, \cite{KangKellLeeWill2009} for an example in the setting of SRBMs and \cite{LatoMiya2014} for an example in the setting of reflecting random
walks.

This paper develops an alternative characterization for a
$d$-dimensional SRBM to have a product form stationary
distribution. This new characterization is based on the geometric
objects associated with the SRBM data $(\Sigma, \mu, R)$. More
specifically, specifying the SRBM data is equivalent to specifying $d+1$
geometric objects: an ellipse $E$ that is specified by $(\Sigma, \mu)$
and $d$ rays that are specified through $R$; the $i$th ray is the unique one that is orthogonal to
$R^{(j)}$ for each $j\neq i$. Ray $i$ intersects the ellipse at a
unique point $\theta^{(i, \ray)}\neq 0$. For each pair $i\neq j$,
$\theta^{(i,\ray)}$ and $\theta^{(j,\ray)}$ span a two-dimensional
hyperplane $\Gamma_{\{i,j\}}$ in $\R^d$ (see \eq{Pij} for its definition). We
draw a line on this hyperplane which goes through the point
$\theta^{(i,\ray)} \in E$ and keeps constant $\theta^{(i,\ray)}_{i}$
in its $i$-th coordinate. This line either is tangent to the
ellipse $E$ at $\theta^{(i,\ray)} \in E$ or intersects the ellipse $E$
at another point.  We denote this point by $\theta^{ij(i,\ray)}$ which
is identical with $\theta^{(i,\ray)}$ if the line is tangent to $E$, and
refer to it as a symmetry point of $\theta^{(i,\ray)}$. Similarly, 
one defines $\theta^{ij(j,\ray)}$ to be the symmetry of
$\theta^{(j,\ray)}$ on the hyperplane $\Gamma_{\{i,j\}}$. 

We prove in \thr{main} that the SRBM has a product form stationary
distribution if and only if $R$ is a ${\cal P}$-matrix and, for every pair $i\neq j$,
\begin{equation}
  \label{eq:symmetryEequal}
  \theta^{ij(i,\ray)} =   \theta^{ij(j,\ray)}.
\end{equation}
\fig{ellipse-ray-1} gives an example illustrating points $\theta^{(1,\ray)}$
and $\theta^{(2,\ray)}$ on the ellipse and their symmetry points $\theta^{12(1,\ray)}$ and $\theta^{12(2,\ray)}$ when $d=2$. 
\begin{figure}[h]
 	\centering
	\includegraphics[height=5.2cm]{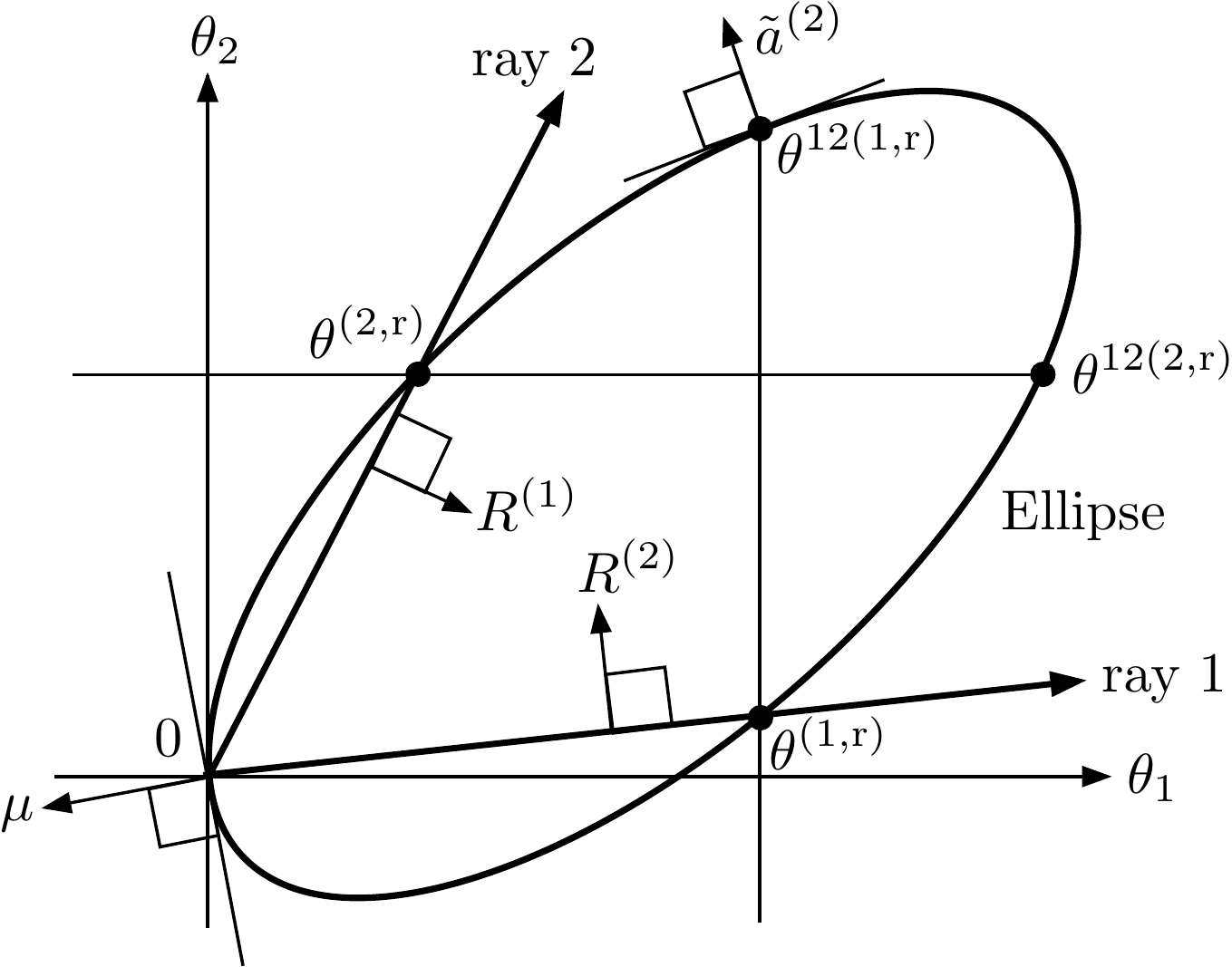} \hspace{2ex}
	\includegraphics[height=5.2cm]{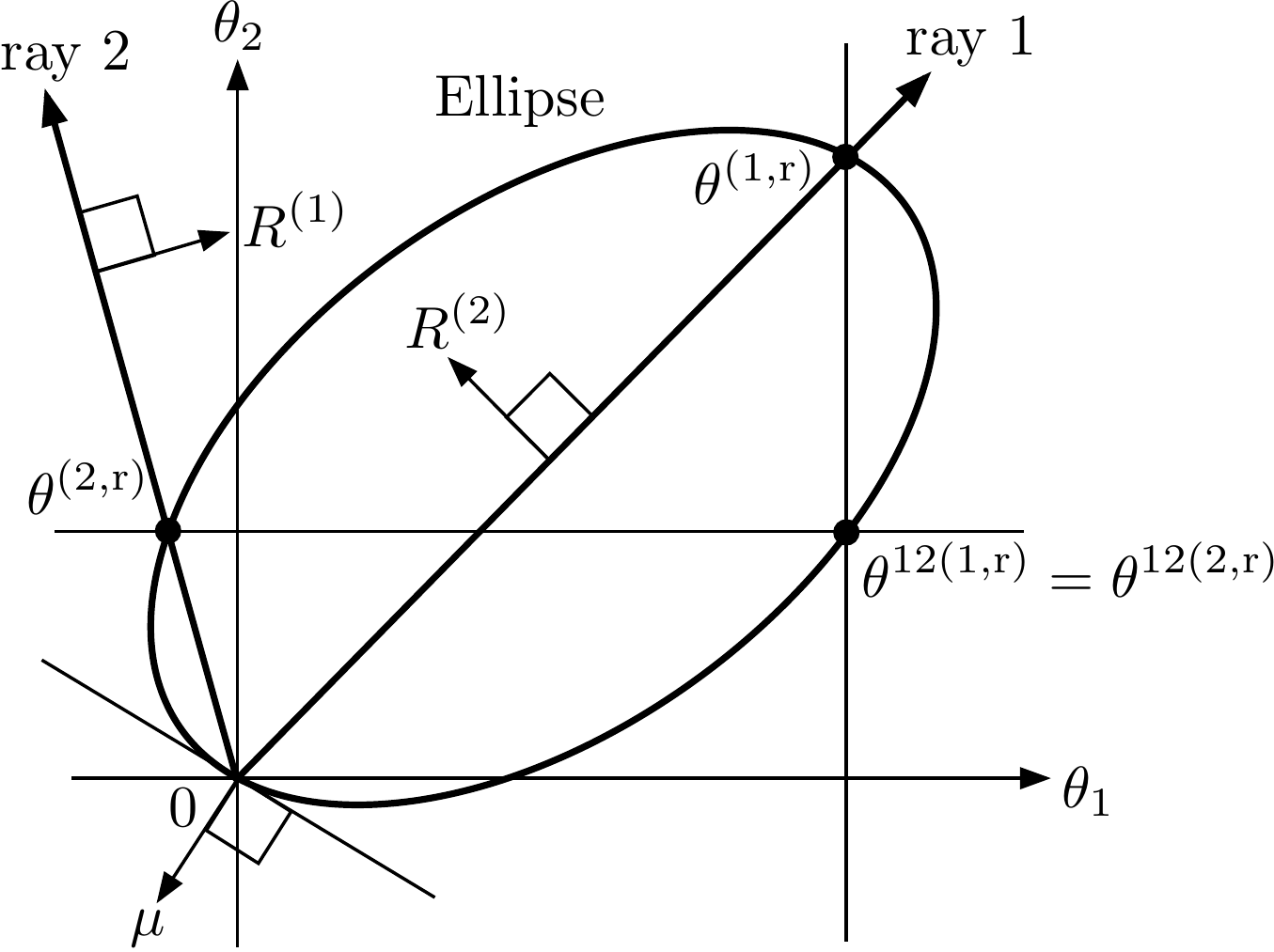}
	\caption{Ellipse, rays and four points to be defined}
	\label{fig:ellipse-ray-1}
\end{figure} 
\thr{main} generalizes the two-dimensional result which is proved in 
\cite{DaiMiya2013}: assume $R$ is a ${\cal P}$-matrix; a two-dimensional SRBM has a product form
stationary distribution if and only if 
\begin{equation}
  \label{eq:2dcharacterization}
  \theta^{12(1,\ray)} = \theta^{12(2,\ray)}.
\end{equation}
We show that those geometric objects on the hyperplane $\Gamma_{\{i,j\}}$
correspond to a two-dimensional SRBM, and we can characterize the the
product form condition of the $d$-dimensional SRBM through the
two-dimensional SRBMs. Interestingly, this simultaneously shows that,
if $R$ is a ${\cal P}$-matrix, then $d$ random
  variables having
the stationary distribution of the SRBM are independent if and only if
each pair of them are independent (see \cor{2}). 

In recent years, there has been an increasing interest in developing
explicit expressions for the tail
asymptotic delay rate of the stationary distribution. However, results are limited for $d=2$ (e.g., see \cite{DaiMiya2011, DaiMiya2013}). There are some studies for $d \ge 3$, but partial results are only available under very restrictive conditions (e.g., see \cite{LianHase2013}). We hope the present geometric interpretations of the product form will make a new step for studying the stationary distribution of a higher dimensional SRBM. We discuss two topics related to this.

The first topic  is about approximation for the stationary distribution. Characterization (\ref{eq:symmetryEequal}) has a potential to allow
one to develop new product form based approximations for the
performance analysis of a general $d$-dimensional SRBM. See \cite{LatoMiya2014} for an example of incorporating tail asymptotics
into product form  approximations. 

The second topic is about a variational problem (VP) associated
with the SRBM. VP is an important, difficult class of problems that
are closely related to the large deviations theory of SRBMs.  See, for
example, \cite{Maje1996,Maje1998,Maje1998a} for the connection between
large deviations and VPs associated with SRBMs. Except for papers
\cite{LianHase2013, KharYaacTahaBich2012}, there has been not
much progress in solving VPs in $d\ge 3$ dimensions.
When $d=2$, \cite{AvraDaiHase2001} shows that the entrance velocities
$\tilde a^{(1)}$ and $\tilde a^{(2)}$ from the first and second
boundary, respectively, play a key role in obtaining the optimal paths
of a VP; see also \cite{HarrHase2009}.  In \cite{DaiMiya2013}, the authors show that 
\begin{displaymath}
  \tilde a^{(2)} = \mu + \Sigma \theta^{12(1,\ray)}.
\end{displaymath}
Namely, the entrance velocity $\tilde a^{(2)}$ from the second boundary
(the $x$ axis) is equal to the outward normal direction of the ellipse
$E$ at the symmetry point $\theta^{12(1,\ray)}$. An analogous
formula holds for $\tilde a^{(1)}$.

For a $d$-dimensional SRBM with a product form stationary
distribution, we have the set of the two dimensional SRBMs which are
used to characterize the product form. These two dimensional SRBMs may
be useful to find the optimal path for the VP because we can apply the
results in \cite{AvraDaiHase2001}. However, we also need to consider
higher dimensional versions of the entrance velocities. This topic will be
discussed using an example, and we conjecture the optimal path for a
three-dimensional product form SRBM. 

This paper consists of five sections. In \sectn{stationary}, we
introduce the basic geometric objects and derive the basic adjoint
relationship (BAR) using the moment generating functions. We also
derive a BAR in quadratic form that characterizes the existence of  
a product form stationary distribution. This characterization is the
foundation of our 
analysis. In \sectn{main}, we introduce the projection idea from the
$d$-dimensional problem to two-dimensional ones and present
our main theorem, \thr{main}.  In \sectn{proof}, we give a detailed proof
of the \thr{main}. In Section \ref{sect:example}, we discuss SRBMs 
arising from tandem queues and the optimal path for some multi-dimensional VPs.

We will use the following notation unless otherwise stated.
\begin{table}[htdp]
\label{tab:primitive}
\begin{center}
\begin{tabular}{ll}
  $J$ & $\{1,2,\ldots, d\}$ \\
  $T^{(i)}$ & the $i$-th column of a square matrix $T$\\
  $T^{ij}$ & 2-dimensional principal matrix composed of the $i$-th and $j$-th rows of $T$\\
  $x^{ij}$ & $(x_{i}, x_{j})^T \in \dd{R}^{2}$ for $x \in \dd{R}^{d}$\\
  $x_{A}$ for $A \subset J$ & the $d$-dimensional vector whose $i$th entry is $x_{i}$ for $i \in A$ and the others zero\\
  $x^{A}$ for $A \subset J$ & $|A|$-dimensional vector $(x_i: i\in A)$\\
  $\br{x, y}$ & $\sum_{i=1}^{d} x_{i} y_{j}$ for $x, y \in \dd{R}^{d}$
\end{tabular}
\caption{A summary of basic notation}
\end{center}
\label{notation 1}
\end{table}%

\section{The stationary distribution and its product form characterization}
\label{sect:stationary}

We assume that $\Sigma$ is positive definite and $R$ is
completely-${\cal S}$ so that $Z$ exists. They, together with the drift
$\mu$, constitute the primitive data of the SRBM. We first describe
them in terms of $d$-dimensional polynomials, which are defined as 
\begin{eqnarray*}
 && \gamma(\theta)=-\frac{1}{2}\langle\theta, \Sigma\theta\rangle-\br{\mu, \theta}, \qquad \theta \in \dd{R}^{d},\\
 && \gamma_{i}(\theta) = \brb{R^{(i)}, \theta}, \qquad \theta \in \dd{R}^{d}, \quad i \in J,
\end{eqnarray*}
where $R^{(i)}$ is the $i$th column of the reflection matrix $R$. Obviously, those polynomials uniquely determine the primitive data, $\Sigma$, $\mu$ and $R$. Thus, we can use those polynomials to discuss everything about the SRBM instead of the primitive data themselves.

Assume the SRBM has a stationary distribution. The stationary
distribution must be unique.
Our first tool is the stationary equation that characterizes
the stationary distribution.  For this, we first introduce
the boundary measures for a distribution $\pi$ on $(\dd{R}_{+}^{d},
\sr{B}(\dd{R}_{+}^{d}))$, where $\sr{B}(\dd{R}_{+}^{d})$ is the Borel
$\sigma$-field on $\dd{R}_{+}^{d}$. They are defined as
\begin{eqnarray*}
\nu_{i}(B)=\E_{\pi} \left[ \int_0^1  1\{ Z(t) \in B\} dY_{i}(t)\right], \qquad B \in \sr{B}(\dd{R}_{+}^{d}), \quad i\in J.
\end{eqnarray*}
The stationary equation is  in terms of moment
generating functions, which  are defined as  
\begin{eqnarray*}
  \varphi(\theta) = \E_{\pi}[ e^{\langle \theta, Z(0)\rangle}],\qquad
  \varphi_{i}(\theta) = \E_{\pi} \left[ \int_0^1  e^{\langle
      \theta, Z(t)\rangle}dY_{i}(t)\right], \quad i\in J,
\end{eqnarray*}
where $\E_\pi$ is the expectation operator when $Z(0)$ is subject to the distribution $\pi$.

Because for each $i\in J$, $Y_{i}(t)$ increases only when $Z_{i}(t)=0$, one has $\varphi_{i}(\theta)$ depends on
$\theta_{J\setminus\{i\}}$ only, where $\theta_{A}$ for $A \subset J$
is the $d$-dimensional vector whose $i$th entry is identical with that
of $\theta$ for $i \in A$ and the entry is zero for $i \in J \setminus
A$. Therefore,  
\begin{displaymath}
\varphi_{i}(\theta)=\varphi_{i}\bigl(\theta_{J\setminus\{i\}}\bigr).
\end{displaymath}

The following lemma is identical to Lemma 1 in \cite{DaiMiyaWu2013}.
We state it  here for easy reference.

\begin{lemma} \label{lem:key}
(a) Assume $\pi$ is the stationary distribution of a $(\Sigma, \mu,
R)$-SRBM.  For $\theta\in \R^d$, $\varphi(\theta)<\infty$ implies
$\varphi_i(\theta)<\infty$ for $i\in J$. Furthermore,
\begin{eqnarray}
    \label{eq:key 1}
    \gamma(\theta) \varphi(\theta)= \sum_{i=1}^d \gamma_{i}(\theta)
    \varphi_{i}\bigl(\theta\bigr),
\end{eqnarray}
holds for $\theta \in \dd{R}^{d}$ such that $\varphi(\theta) < \infty$.
(b) Assume that $\pi$ is a probability measure on $\R^d_+$ and that
$\nu_i$ is a positive finite measure whose support is contained in
$\{x\in\R^d_+: x_i=0\}$ for $i\in J$. Let $\varphi$ and $\varphi_i$
be the moment generating functions of $\pi$ and $\nu_i$,
respectively. If $\varphi$, 
$\varphi_1$, $\ldots$, $\varphi_d$ satisfy (\ref{eq:key 1}) for each $\theta\in
\R^d$ with $\theta\le 0$, then $\pi$ is the stationary distribution
and $\nu_i$ is the boundary measure of the associated SRBM on $\{x\in\R^d_+: x_i=0\}$.
\end{lemma}
  Equation
\eq{key 1} is the moment generating function version of the
standard basic adjoint relationship (BAR) that was first derived 
in \cite{HarrWill1987}; for the standard BAR, see also equation (7) of
\cite{DaiHarr1992}.
Part (a) is now standard, following Proposition~3 of
  \cite{DaiHarr1992} and Lemma~4.1 of 
\cite{DaiMiya2011}. For part (b), one can follow a standard
  procedure (see Proposition~\ref{pro:key 1} in Appendix
  \ref{sec:equiv-two-vers}) to  argue that  Equation 
  \eq{key 1}  is equivalent to the standard BAR. The rest of part (b)
is implied by \cite{DaiKurt1994}.

From now on, we always assume that  $\pi$ is the stationary
distribution of the SRBM unless otherwise is stated. 
It follows from  \cite{HarrWill1987a} that the stationary distribution
of SRBM, when exists,  has a density. We use $\zeta(y)$ to denote
the stationary density of $d$-dimensional SRBM. 
Thus, the stationary distribution has product form if and only if
\begin{eqnarray}
  \label{eq:density}
   \zeta(y)=\prod_{i=1}^d \zeta_{i}(y_{i}),
\end{eqnarray}
where $\zeta_{i}$'s are the marginal densities of $\zeta$. It follows
from the first Theorem  in Section 9 of \cite{HarrWill1987} on page 107 that when the stationary density is of
product form in \eq{density}, each $\zeta_{i}$ must be exponential. 
Thus, $d$ -dimensional SRBM has a product form if and only if there
exists a $d$-dimensional vector $\alpha>0$ such that 
\begin{eqnarray}
  \label{eq:density1}
   \zeta(y)=\prod_{i=1}^d \alpha_{i}e^{-\alpha_{i}y_{i}}.
\end{eqnarray}

It is shown in \cite{HarrWill1987a} that, under the skew
symmetry condition \eq{skew symmetric},  
the SRBM has a product-form stationary density in \eq{density1} and $\alpha$ is given by 
\begin{eqnarray}
\label{eq:alpha}
  \alpha = -2\diag(\Sigma)^{-1}\diag(R)R^{-1}\mu.
\end{eqnarray}

In this paper, we provide alternative characterizations for the
product form in terms of a set of two-dimensional SRBMs. For each
two-dimensional SRBM,  a geometric interpretation for the product form
condition is derived in \cite{DaiMiya2013}, and therefore the
necessary and sufficient condition of this paper has also geometric
interpretation. 

The following is a key lemma to characterize the product form of SRBM which will be used repeatedly in
 this paper.

\begin{lemma}
\label{lem:characterization}
Assume $R$ is completely-${\cal S}$ and condition (\ref{eq:stability})
is satisfied.
The $d$-dimensional SRBM has a product form stationary distribution
with its density in \eq{density1}  for some  $\alpha =(\alpha_1,
\ldots, \alpha_d)^T>0$ 
if and only if for some positive constants $C_1$, $\ldots$, 
  $C_d$
\begin{eqnarray}
\label{eq:characterization}
\gamma(\theta)=\sum_{i=1}^d
C_i\gamma_{i}(\theta)(\alpha_{i}-\theta_{i}) \text{ for } \theta\in \R^d.
\end{eqnarray}
Furthermore, if \eq{characterization} holds, then
$C_i=\Sigma_{ii}/(2R_{ii})$, $\alpha$ is given in (\ref{eq:alpha}),
and  $\gamma(\alpha) = 0$.
\end{lemma}
\begin{remark}
The above lemma can  also  be used to
show that \eq{skew symmetric} holds if and only if the stationary
distribution of SRBM has a product form. See Appendix \ref{sec:skew}.
\end{remark}
\begin{proof}

Assume that SRBM has a product form stationary density as in
\eq{density1} for some  $\alpha =(\alpha_1, \ldots, \alpha_d)^T>0$.
Then following from \cite{HarrWill1987}, we know that the boundary measure $\nu_{i}$ has density:
\begin{eqnarray}
\label{eq:boundary}
\zeta^i(y)=\frac{\Sigma_{ii}}{2R_{ii}}\alpha_{i}\prod_{k\not=i}e^{-\alpha_ky_k}
\text{ for } i\in J \text{ and } y\in \R^d_+.
\end{eqnarray}
Following the above equations, we have:
\begin{eqnarray}
\label{eq:phi}
\varphi(\theta)=\prod_{i=1}^{d}\frac{\alpha_{i}}{\alpha_{i}-\theta_{i}}
\quad \text{ for } \theta < \alpha.
\end{eqnarray}
\begin{eqnarray}
\label{eq:phii}
\varphi_{i}(\theta)=\frac{\Sigma_{ii}}{2R_{ii}}\alpha_{i}\prod_{k\not=i}\frac{\alpha_k}{\alpha_k-\theta_k}
\quad \text{ for } \theta< \alpha.
\end{eqnarray}
Substituting these $\varphi(\theta)$ and $\varphi_{i}(\theta)$ into \eq{key 1} of \lem{key}, we have \eq{characterization} for any
$\theta<\alpha$. In particular,  \eq{characterization} holds for
infinitely many $\theta$'s. Since both sides of  \eq{characterization} 
are quadratic in $\theta$,  \eq{characterization} holds for all
$\theta\in \R^d$.
 Conversely, if there exists an $\alpha>0$ such that
\eq{characterization} holds,  one can define $\varphi(\theta)$ and
$\varphi_{i}(\theta)$ as in \eq{phi} and \eq{phii}.  They satisfy
\eq{key 1}. Then the moment generating functions of the stationary density \eq{density1} and boundary
densities \eq{boundary} satisfy \eq{key 1} for $\theta \le 0$. So according to part (b) of \lem{key}, the SRBM must have \eq{density1} as its
stationary density. 

Assume (\ref{eq:characterization}) holds. Comparing the coefficients of $\theta_i^2$ on both sides of
(\ref{eq:characterization}), we obtain
$C_i={\Sigma_{ii}}/{2R_{ii}}$ for $1 \le i \le d$. By comparing the
coefficients of $\theta_i$ for $i=1, \ldots, d$, we have
(\ref{eq:alpha}). 
The fact that $\gamma(\alpha)=0$ in the last statement is easily verified by substituting $\theta = \alpha$ into \eq{characterization}.
\end{proof}

\section{Geometric objects and main results}
\label{sect:main}
We consider a $d$-dimensional SRBM having data $(\Sigma, \mu, R)$.
We assume that $\Sigma$ is positive definite,   $R$ is
completely-${\cal S}$, and condition
\eq{stability} holds.
From  BAR \eq{key 1}, one can imagine that the tail decay rate of the
stationary distribution would be related to the $\theta$ at which the
coefficients of $\varphi(\theta)$ and
$\varphi_{i}(\theta_{J\setminus\{i\}})$ in \eq{key 1} becomes zero. Thus, we
introduce the following geometric objects:
\begin{eqnarray*}
  E = \{\theta \in \dd{R}^{d}; \gamma(\theta) = 0\}, \qquad \Gamma_{\{i\}} =
  \cap_{k \in J \setminus \{i\}}\{\theta \in \dd{R}^{d};
  \gamma_{k}(\theta) = 0\}, \quad i \in J. 
\end{eqnarray*}
These geometric objects are well defined even when the SRBM does not
have a stationary distribution. 
The object $E$ is an ellipse in $\R^d$.  Since $R$ is
invertible and $\theta \in \Gamma_{\{i\}}$ implies that $\br{\theta, R^{(k)}} =
0$ for $k \ne i$, $\Gamma_{\{i\}}$ must be a line going through the origin.
Clearly, for each $i$, $\Gamma_{\{i\}}$ intersects the ellipse $E$ at most
two points, one of which is the origin. We denote its non-zero
intersection by $\theta^{(i,\ray)}$ if it exists. Otherwise, let
$\theta^{(i,\ray)} = 0$. The following lemma shows that the latter is
impossible by 
giving an explicit formula for $\theta^{(i,\ray)}$.
For that let $B=(R^{-1})^T$ and $B^{(i)}$ be the $i$th column of
$B$. Equivalently, the transpose of ${B^{(i)}}$ is the $i$th row of $R^{-1}$.

\begin{lemma} 
\label{lem:ray}
For each $i\in J$, 
\begin{equation}
  \label{eq:thataray}
  \theta^{(i,\ray)} = \Delta_i B^{(i)},
\end{equation}
where 
\begin{equation}
  \label{eq:Deltai}
  \Delta_i = -\frac{2 \langle \mu, B^{(i)} \rangle}{\langle B^{(i)},
    \Sigma B^{(i)}\rangle}>0.
\end{equation}
\end{lemma}
\begin{proof}
Because $R^{-1}R=I$, we have $B^{(i)}\neq 0$ and $B^{(i)}\in
\Gamma(i)$. Therefore, (\ref{eq:thataray}) holds for some $\Delta_i\in
\R$. Since $\theta^{(i,\ray)}\in E$, we have $\gamma(\Delta_i
B^{(i)})=0$, from which we have the equality in (\ref{eq:Deltai}). If
(\ref{eq:stability}) holds, we have $\langle \mu, B^{(i)} \rangle <0$,
from which the inequality in (\ref{eq:Deltai}) holds.
\end{proof}

Let 
\begin{equation}
\label{eq:A}  
  A =  (\theta^{(1,\ray)}, \ldots, \theta^{(d,\ray)})
\quad \text{ and } \quad
  A^{ij} = \left(\begin{array}{cc} \theta^{(i,\ray)}_{i} & \theta^{(j,\ray)}_{i} \\
  \theta^{(i,\ray)}_{j} & \theta^{(j,\ray)}_{j} \end{array}\right)
\text{ for } 1\le i < j \le d.
\end{equation}
By Lemma \ref{lem:ray},  $A=B\Delta$, where 
$\Delta =\text{diag}(\Delta_1, \ldots, \Delta_d)$.
Clearly, $A^{ij}$ is a $2 \times 2$ principal sub-matrix of $A$. We
let 
\begin{equation}
  \label{eq:cij}
c_{ij}=\det(A^{ij}).
\end{equation}

For each pair $(i, j)$ with $i,j\in J$ and $i<j$, we define the
two-dimensional hyperplane in $\R^d$:
\begin{eqnarray}
\label{eq:Pij}
    \Gamma_{\{i,j\}} =\cap _{k\in J\setminus \{i,j\}} \{\theta \in \R^d;
  \gamma_k(\theta)=0\}.
\end{eqnarray}
We then define a mapping $f^{ij}$ from $\dd{R}^{2}$ to $\Gamma_{\{i,j\}}\subset\R^d$ such that $f^{ij}(\theta_{i}, \theta_{j}) = \theta$ for any $\theta \in \Gamma_{\{i,j\}}$. Sometimes, we write $f^{ij}(\theta_i, \theta_j)$ as $f^{ij}(\theta^{ij})$ where $\theta^{ij}=(\theta_i, \theta_j)^T$. The following lemma confirms that $f^{ij}$ is well defined if $c_{ij} \not=0$. Its proof is given in Appendix
\ref{sec:proofs-lemmas}.

\begin{lemma}
  \label{lem:hyper-2d}
For each $i \ne j \in J$, if $c_{ij}=\det(A^{ij})\neq 0$, then the mapping $f^{ij}$ defined above uniquely exists.
\end{lemma}
Since both points $\theta^{(i,\ray)}$ and $\theta^{(j
,\ray)}$  are on $\Gamma_{\{i,j\}}$ and they are linearly independent, for $z^{ij} \equiv (z_{i}, z_{j})^T \in \dd{R}^{2}$, $f^{ij}(z^{ij})$ 
is a linear combination of $\theta^{(i,\ray)}$ and
$\theta^{(j,\ray)}$.
Indeed, one can check that
\begin{equation}
  \label{eq:fij}
  f^{ij}(z^{ij}) = \frac{1}{c_{ij}} \Bigl(
\bigl ( \theta^{(i,\ray)}_i z_i -\theta^{(j,\ray)}_i z_j \bigr )
\theta^{(i,\ray)} 
+ \bigl ( -\theta^{(i,\ray)}_j z_i +\theta^{(j,\ray)}_j z_j \bigr )
\theta^{(j,\ray)} \Bigr).
\end{equation}
We remark that $c_{ij}$ in (\ref{eq:cij}) can be zero even if $R$ is completely-${\cal
  S}$ and condition (\ref{eq:stability}) is satisfied; see Example
\ref{exa:pmatrix} in Appendix~\ref{sec:Examples}.

The intersection $E\cap \Gamma_{\{i,j\}}$ of the ellipse and the hyperplane  is an
 ellipse  on  hyperplane $\Gamma_{\{i,j\}}$.  Both
$\theta^{(i,\ray)}$ and $\theta^{(j,\ray)}$ are on $E\cap \Gamma_{\{i,j\}}$.
Now we define two points $\theta^{ij(i,\ray)}$ and $
\theta^{ij(j,\ray)}$  that are symmetries of 
$\theta^{(i,\ray)}$ and $\theta^{(j,\ray)}$ on $E\cap \Gamma_{\{i,j\}}$,
respectively.  
If 
\begin{equation}
  \label{eq:extremeCondition}
  \theta^{(i,\ray)}_i = \argmax\{ \theta_i: \theta \in E\cap \Gamma_{\{i,j\}}\},
\end{equation}
define $\theta^{ij(i,\ray)}=\theta^{(i,\ray)}$. 
Otherwise, define $\theta^{ij(i,\ray)}$ to be the unique $\theta\in
\R^d$ that satisfies
\begin{equation}
  \label{eq:tildetheta}
  \theta \in E\cap \Gamma_{\{i,j\}}, \quad  \theta_i =
  \theta^{(i,\ray)}_i, \quad
\text{ and } \quad \theta\neq \theta^{(i,\ray)}.
\end{equation}
The point $\theta^{ij(i,\ray)}$ is well defined because of the
following lemma, which will be proved in Appendix~
\ref{sec:proofs-lemmas}. 
\begin{lemma}
  \label{lem:uniqueSymmetry}
If (\ref{eq:extremeCondition}) holds, the quadratic equation 
\begin{equation}
  \label{eq:quadraticzj}
\gamma\Bigl(\frac{1}{c_{ij}} \Bigl(
\bigl ( \theta^{(i,\ray)}_i \theta^{(i,\ray)}_i-\theta^{(j,\ray)}_i z_j \bigr )
\theta^{(i,\ray)} 
+ \bigl ( -\theta^{(i,\ray)}_j \theta^{(i,\ray)}_i
+\theta^{(j,\ray)}_j z_j \bigr ) 
\theta^{(j,\ray)} \Bigr)\Bigr)=0  
\end{equation}
has a unique (double) solution $z_j=\theta^{(i,\ray)}_j$. Otherwise,
(\ref{eq:quadraticzj}) has two solutions $z_j'=\theta^{(i,\ray)}_j$ and
$z_j\neq \theta^{(i,\ray)}_j$. 
\end{lemma}
Let $z_j$  be the solution in Lemma \ref{lem:uniqueSymmetry}. Then the
symmetry of $\theta^{(i,\ray)}$ is equal to 
\begin{displaymath}
  \theta^{ij(i,\ray)} = f^{ij}\bigl(\theta^{(i,\ray)}_i, z_j\bigr).
\end{displaymath}
Similarly we define $\theta^{ij(j,\ray)}=\theta^{(j,\ray)}$ if 
$\theta^{(j,\ray)}_j=\argmax\{\theta_j: \theta\in E\cap
\Gamma_{\{i,j\}}\}$. Otherwise, it is defined to the unique $\theta\in \R^d$
that satisfies $\theta \in  E\cap \Gamma_{\{i,j\}}$, $\theta_j =
\theta^{(j,\ray)}_j$, and $\theta\neq \theta^{(j,\ray)}$.

We need one more lemma, which will be proved in Appendix \ref{sec:proofs-lemmas}.

\begin{lemma}
\label{lem:pmatrix}
(a) $R^{-1}$ is a $\mathcal{P}$-matrix if $R$ is a $\mathcal{P}$-matrix.\\
(b) Assume that $R^{-1}\mu<0$, $c_{ij}$ in (\ref{eq:cij}) is
  positive if $R$ is $\mathcal{P}$-matrix. \\ 
(c) If $d$-dimensional SRBM has a product form stationary distribution, then $R$ is a $\mathcal{P}$-matrix, and therefore $c_{ij} > 0$.
\end{lemma}

\begin{theorem}
  \label{thr:main}
  Assume that $\Sigma$ is a $d\times d$ positive definite matrix and
  $\mu$ is a $d$-dimensional vector.  Assume that $R$ is a $d\times d$
  completely-$S$ matrix, and $(R, \mu)$ satisfies (\ref{eq:stability}). 
(a) The $(\Sigma, \mu, R)$-SRBM has a product form stationary
distribution if and only if $R$ is a ${\cal P}$-matrix and 
\begin{equation}
  \label{eq:identicalSymmetry}
 \theta^{ij(i,\ray)}  =   \theta^{ij(j,\ray)}  \quad
\text{ for each  } 1\le i< j\le d.
\end{equation}
(b) If for each $1\le i < j\le d$, $A^{ij}$ in (\ref{eq:A}) is a ${\cal P}$-matrix and
(\ref{eq:identicalSymmetry}) is satisfied, then 
the $(\Sigma, \mu, R)$-SRBM has a product form stationary
distribution.
\end{theorem}

The proof of \thr{main} will be given in Section \ref{sect:proof}. For that, we define 
\begin{equation}
\label{eq:Sigmastar}
\Sigma^*=A^{\rs{t}} \Sigma A\quad \text{and} \quad\mu^{*}=A^{\rs{t}} \mu.
\end{equation} 
The main idea in the proof is to prove
that the $d$-dimensional SRBM has a product form stationary distribution if
and only if for each $1\le i< j\le d$, the two-dimensional $(\tilde
\Sigma^{ij}, \tilde \mu^{ij}, \tilde R^{ij})$-SRBM is well defined and
has a product form 
stationary distribution, where
\begin{eqnarray}
\label{eq:2d Sigma}
&& \tilde{\Sigma}^{ij} = ((A^{ij})^{\rs{t}})^{-1} (\Sigma^{*})^{ij}
 (A^{ij})^{-1},\\
&& \tilde{\mu}^{ij} = ((A^{ij})^{\rs{t}})^{-1} (\mu^{*})^{ij}, \quad\text{and}\quad\tilde{R}^{ij} = ((A^{ij})^{\rs{t}})^{-1} \text{diag}(\Delta_i,
\Delta_j).
\end{eqnarray}

In the following corollary, we set
\begin{equation}
  \label{eq:1}
  \tau=(\theta^{(1,\ray)}_1, \ldots, \theta^{(d,\ray)}_d)^T, \quad
  \text{ and } \quad
  \tau^{ij} = (\tau_i, \tau_j)^T.
\end{equation}

\begin{corollary}
\label{cor:main}
Under the assumptions of \thr{main}, the $(\Sigma, \mu, R)$-SRBM has a product form stationary
distribution if and only if and for each $i,j\in J$ with $i<j$, $A^{ij}$ is a ${\cal P}$-matrix and
\begin{eqnarray}
  \label{eq:tauijongamma}
&& \gamma\bigl (  f^{ij}(\tau^{ij}) \bigr )=0, \\
&& \theta^{(i,\ray)}_j \neq \tau_j \text{ if } \theta^{(i,\ray)}_i \neq
 \argmax\{\theta_i: \theta\in E\cap \Gamma_{\{i,j\}}\}, \text{ and }  \label{eq:thetaithetamax}\\
&& \theta^{(j,\ray)}_i \neq \tau_i \text{ if } \theta^{(j,\ray)}_j \neq
 \argmax\{\theta_j: \theta\in E\cap \Gamma_{\{i,j\}}\}. \label{eq:thetajthetamax}
\end{eqnarray}
\end{corollary}

\cor{main} will be proved shortly below. One may wonder how the
two-dimensional $(\tilde{\Sigma}^{ij}, \tilde{\mu}^{ij}, 
\tilde{R}^{ij})$-SRBM is related to the two-dimensional marginal process
$\{(Z_{i}(t),Z_{j}(t)), t\ge 0\}$. 
The next corollary answers this question. The proof of this corollary
will be given at the end of Section~\sect{proof}. To state the
corollary, let $Z(\infty)$ be a random vector that has the
distribution 
to the stationary distribution of the $d$-dimensional SRBM $Z=\{Z(t),
t\ge 0\}$. $Z_{i}(\infty)$ and $Z_{j}(\infty)$ are the $i$th and $j$th
components of $Z(\infty)$ for each $i \ne j \in J$.

\begin{corollary}
\label{cor:2}
Under the assumptions of \thr{main}, we have the following facts. (a) The $d$-dimensional SRBM $Z=\{Z(t), t\ge 0\}$ has a product form stationary distribution if and only if, for each $i \neq j\in J$, $\tilde{R}^{ij}$ is a $\sr{P}$-matrix and
the two-dimensional $(\tilde\Sigma^{ij}, \tilde \mu^{ij}, \tilde   R^{ij})$-SRBM has a product form stationary distribution, which is identical to the distribution of $(Z_i(\infty), Z_j(\infty))$. (b) If $R$ is a ${\cal P}$-matrix, then $Z_{1}(\infty), Z_{2}(\infty), \ldots, Z_{d}(\infty)$ are independent if and only if, for each $i \neq j\in J$, $Z_i(\infty)$ and $Z_j(\infty)$ are independent.
\end{corollary}

\begin{proof}[Proof of Corollary \ref{cor:main}]

We first prove the necessity. Assume $i<j$. Since $R$ is a $P$-matrix, $R^{-1}$ is a $P$-matrix by \lem{pmatrix}. This implies that $A$ is a $P$-matrix, and therefore $A^{ij}$ is a $P$-matrix.

 Recall that 
$\theta^{ij(i,\ray)}_i=\theta^{(i,\ray)}_i=\tau_i$ and 
$\theta^{ij(j,\ray)}_j=\theta^{(j,\ray)}_j=\tau_j$.
Thus, (\ref{eq:identicalSymmetry}) is equivalent to 
\begin{eqnarray}
  \label{eq:thetataui}
&&  \theta^{ij(i,\ray)} = f^{ij}(\tau^{ij}), \\
  \label{eq:thetatauj}
&&  \theta^{ij(j,\ray)} = f^{ij}(\tau^{ij}).
\end{eqnarray}
Assume the product form stationary
distribution. Then, by part (a) of \thr{main},
(\ref{eq:identicalSymmetry}) holds. As a consequence,  both
(\ref{eq:thetataui}) and (\ref{eq:thetatauj}) hold.
Since $\theta^{ij(i,\ray)}$ is on the ellipse, (\ref{eq:thetataui})
implies  (\ref{eq:tauijongamma}).  We now prove (\ref{eq:thetaithetamax})
must hold. Assume that 
\begin{equation}
  \label{eq:2}
\theta^{(i,\ray)}_i \neq
 \argmax\{\theta_i: \theta\in E\cap \Gamma_{\{i,j\}}\}.
\end{equation}
Suppose on the contrary  that $\theta^{(i,\ray)}_j=\tau_j$. This implies
that $\theta^{ij(i,\ray)}=\theta^{(i,\ray)}$, which contradicts the
condition $\theta^{ij(i,\ray)}\neq\theta^{(i,\ray)}$ in the definition
of (\ref{eq:tildetheta}). Similarly, we can prove
(\ref{eq:thetajthetamax}) holds. This proves the necessity.

Now we prove the sufficiency. Let $i,j\in J$ with $i<j$.  Assume $A^{ij}$ is
a ${\cal P}$-matrix and  
(\ref{eq:tauijongamma})-(\ref{eq:thetajthetamax}) hold. Then
(\ref{eq:tauijongamma}) and (\ref{eq:thetaithetamax}) imply
(\ref{eq:thetataui}), and 
(\ref{eq:tauijongamma}) and (\ref{eq:thetajthetamax}) imply
(\ref{eq:thetatauj}). Thus, (\ref{eq:identicalSymmetry}) holds. 
It follows from part (b) of \thr{main} that the SRBM has a product
form stationary distribution. 
\end{proof}

\begin{remark}
In the two-dimensional case, when $\tau_1\neq \theta^{(2,\ray)}_1$ and
$\tau_2\neq \theta^{(1,\ray)}_2$, the SRBM has a product form
stationary distribution if and only if 
the point $\tau$ is on the ellipse, i.e.,
$\gamma(\tau)=0$. Example~\ref{exa:non-product form} in Appendix \ref{sec:Examples}
shows that when $d\ge 3$, the condition $\gamma(\tau)=0$ is not
sufficient for a product 
form stationary distribution. 
\end{remark}

We end this section by stating a lemma that will be needed in the
proof of \thr{main}, and proved in Appendix \ref{sec:proofs-lemmas}. To state the following lemma, 
for $z^{ij}=(z_i, z_j)^T\in \R^2$, 
let
\begin{equation}
  \label{eq:gammaij}
  \tilde{\gamma}^{ij}(z^{ij}) = - \frac 12 \brb{z^{ij}, \tilde{\Sigma}^{ij} z^{ij}} - \brb{\tilde{\mu}^{ij}, z^{ij}}.
\end{equation}


\begin{lemma}
  \label{lem:gammaijgamma}
For $z^{ij}=(z_i, z_j)^T\in \R^2$,  $\tilde \gamma^{ij}(z^{ij})=\gamma(f^{ij}(z^{ij}))$.
\end{lemma}


\section{Proof of  \thr{main}}
  \label{sect:proof}

We first prove the necessity in (a) of \thr{main}. Assume that the $(\Sigma, \mu,
R)$-SRBM has a product form stationary distribution. Therefore, for
some $\alpha\in \R^d$ with $\alpha>0$, 
(\ref{eq:characterization}) holds for every $\theta \in \R^d$.

By part (c) of \lem{pmatrix}, $R$ is a ${\cal P}$-matrix. We
now prove (\ref{eq:identicalSymmetry}). 
By Lemma~\ref{lem:hyper-2d}, it suffices to prove that
 for any
$1\le i< j \le d$, 
\begin{equation}
  \label{eq:tildethetaiequalj}
   \theta^{ij(i,\ray)}_i=   \theta^{ij(j,\ray)}_i =
  \alpha_i, \quad \text{ and } \quad 
   \theta^{ij(i,\ray)}_j= \theta^{ij(j,\ray)}_j = \alpha_j,
\end{equation}
 where the condition $c_{ij} \ne 0$ is satisfied by \lem{pmatrix}.
  
To prove (\ref{eq:tildethetaiequalj}), we follow the derivation from
(5.17) to (5.21) of \cite{DaiMiya2013}. 
Observe that $\gamma(\theta^{(i,\ray)})=0$, $\gamma_k(\theta^{(i,\ray)})=0$
for $k\in J\setminus\{i\}$, and $\gamma_i(\theta^{(i,\ray)})\neq 0$;
the latter holds because $R$ is assumed to be invertible in
(\ref{eq:stability}). Plugging 
$\theta=\theta^{(i,\ray)}$  into (\ref{eq:characterization}), we have 
$\alpha_i=\theta^{(i,\ray)}_i$.
Thus,  by the definition of 
$ \theta^{ij(i,\ray)}$, we have 
$ \theta^{ij(i,\ray)}_i=\theta^{(i,\ray)}_i=\alpha_i$. 
We now show that
\begin{equation}
  \label{eq:alphaj}
    \theta^{ij(i,\ray)}_j=\alpha_j.
\end{equation}
To see this, by the definition of $ \theta^{ij(i,\ray)}$, we
have $\gamma( 
\theta^{ij(i,\ray)})=0$, $\gamma_k(
\theta^{ij(i,\ray)})=0$ for $k\in J\setminus\{i,j\}$.
If $ \theta^{ij(i,\ray)} \neq \theta^{(i,\ray)}$, we have $\gamma_j(
\theta^{ij(i,\ray)})\neq 0$.
 Plugging $\theta=
\theta^{ij(i,\ray)}$ into (\ref{eq:characterization}), we have 
\begin{displaymath}
  \gamma_j(
\theta^{ij(i,\ray)})(\alpha_j - 
\theta^{ij(i,\ray)}_j)=0,
\end{displaymath}
from which we conclude that (\ref{eq:alphaj}) holds.
If $ \theta^{ij(i,\ray)} = \theta^{(i,\ray)}$, then
(\ref{eq:extremeCondition}) holds. According to Lemma \ref{lem:gammaijgamma}, $\theta^{(i,\ray)}_i =\argmax\{ z_i: z^{ij}=(z_i, z_j)^T\in \R^2, \tilde
  \gamma^{ij}(z^{ij})=0\}$. At the same time, if (\ref{eq:characterization}) holds, then plugging the definition of $f^{ij}(z^{ij})$ in (\ref{eq:fij}) into (\ref{eq:characterization}), we have
\begin{eqnarray}
\label{eq:gammaijcharacterization}
\lefteqn{ \tilde \gamma^{ij}(z^{ij})=\gamma(f^{ij}(z^{ij}))=C_i\Delta_ic_{ij}^{-1}(\theta^{(j,\ray)}_jz_i-\theta^{(j,\ray)}_iz_j)(\alpha_i-z_i)} \hspace{10ex} \nonumber\\
&& +C_j\Delta_jc_{ij}^{-1}(-\theta^{(i,\ray)}_jz_i+\theta^{(i,\ray)}_iz_j)(\alpha_j-z_j).
 \end{eqnarray}

Taking derivative in the both sides of (\ref{eq:gammaijcharacterization})
with respect to $z_j$, and plugging $(\theta^{ij(i,\ray)}_i, \theta^{ij(i,\ray)}_j)^T$
into the new equation, we again conclude that (\ref{eq:alphaj})
holds as 
\begin{displaymath}
\frac{  \partial \tilde \gamma^{ij}(z_{ij})}{\partial z_j}
|_{z^{ij}=(\theta^{ij(i,\ray)}_i,\theta^{ij(i,\ray)}_j)^T}=C_j\Delta_jc_{ij}^{-1}\theta^{(i,\ray)}_{i}(\alpha_j-\theta^{ij(i,\ray)}_j)
=0.
\end{displaymath}
 Similarly, we can show that 
\begin{displaymath}
   \theta^{ij(j,\ray)}_i = \alpha_i, \quad \text{ and } \quad
 \theta^{ij(j,\ray)}_j = \alpha_j,
\end{displaymath}
thus proving (\ref{eq:tildethetaiequalj}). This concludes the necessity proof.

We note that the sufficiency of (a) in \thr{main} is immediate from (b) and \lem{ray} because $R^{-1}$ is $\sr{P}$-matrix by (a) of \lem{pmatrix}. Thus, it remains only to prove (b). For any fixed pair $1\le i< j\le d$, assume that $A^{ij}$ is a ${\cal P}$-matrix and (\ref{eq:identicalSymmetry}) holds.
 Using Theorem 5.1 of 
\cite{DaiMiya2013}, we would like to  conclude that two-dimensional
 $(\tilde\Sigma^{ij}, \tilde \mu^{ij}, \tilde
R^{ij})$-SRBM has a product form stationary distribution.

Now we prove that (\ref{eq:identicalSymmetry}) implies condition (5.2) in
Theorem 5.1 of \cite{DaiMiya2013}.
 For this,
we define the geometric objects associated with  
the two-dimensional $(\tilde\Sigma^{ij}, \tilde \mu^{ij}, \tilde
R^{ij})$-SRBM. Recall the definition of $\tilde \gamma^{ij}(z^{ij})$ in
(\ref{eq:gammaij}).
  Then, $\tilde{\gamma}^{ij}(z^{ij}) = 0$ defines an ellipse in
  $\R^2$. 
Let 
\begin{equation}
  \label{eq:gammaiji}
\tilde{\gamma}^{ij}_{k}(z^{ij}) = \brb{z^{ij},
  (\tilde{R}^{ij})^{(k)}}, \quad k =i, j,
\end{equation}
where $ (\tilde{R}^{ij})^{(k)}$ is the $k$th column of ${\tilde R}^{ij}$.
Then $\tilde{\gamma}^{ij}_{k}(z^{ij}) = 0$  defines
  a line in $\R^2$ for $k=1,2$.
We next find the non-zero intersection points of the ellipse
$\tilde{\gamma}(z^{ij}) = 0$ and the lines
$\tilde{\gamma}^{ij}_{i}(z^{ij}) = 0$ and
$\tilde{\gamma}^{ij}_{j}(z^{ij}) = 0$, respectively, on
$\dd{R}^{2}$. 
By (\ref{eq:fij}), we have
\begin{eqnarray*}
f^{ij} (\theta^{(i,\ray)}_{i},
  \theta^{(i,\ray)}_{j}) = \theta^{(i,\ray)}, \quad \text{ and } \quad
 f^{ij}(\theta^{(j,\ray)}_{i}, \theta^{(j,\ray)}_{j})=\theta^{(j,\ray)}.
\end{eqnarray*}
Therefore, we can use Lemma~\ref{lem:gammaijgamma} 
and expressions \begin{eqnarray*}
  \tilde{\gamma}^{ij}_{i}(z^{ij}) = c_{ij}^{-1} \Delta_{i} (z_{i}
  \theta^{(j,\ray)}_{j} - z_{j} \theta^{(j,\ray)}_{i}) \quad \text{ and } \quad
  \tilde{\gamma}^{ij}_{j}(z^{ij}) = c_{ij}^{-1} \Delta_{j} ( - z_{i}
  \theta^{(i,\ray)}_{j} + z_{j} \theta^{(i,\ray)}_{i}) 
\end{eqnarray*}
to  verify that  these intersection points are given by
\begin{eqnarray*}
 (\theta^{(i,\ray)}_{i},
  \theta^{(i,\ray)}_{j})^T, \quad \text{ and } \quad
 (\theta^{(j,\ray)}_{i}, \theta^{(j,\ray)}_{j})^T.
\end{eqnarray*}
Define 
\begin{eqnarray*}
  \tilde\theta^{ij(i,\ray)} = ( \theta^{ij(i,\ray)}_{i},
   \theta^{ij(i,\ray)}_{j})^T, \qquad \theta^{*ij(j,\ray)} =
  ( \theta^{ij(j,\ray)}_{i},  \theta^{ij(j,\ray)}_{j})^T.
\end{eqnarray*}
Then $f^{ij}(  \tilde\theta^{ij(i,\ray)})= \theta^{ij(i,\ray)}$.
By Lemma \ref{lem:gammaijgamma},  we have
$\tilde \gamma^{ij}(\tilde \theta^{ij(i,\ray)})=
\gamma( \theta^{ij(i,\ray)})=0$,  where
the latter equality 
follows from the definition of $ \theta^{ij(i,\ray)}$. It
follows from Lemma \ref{lem:gammaijgamma} that
(\ref{eq:extremeCondition}) holds if and 
only if 
\begin{equation}
  \label{eq:twodimthetamax}
  \theta^{(i,\ray)}_i =\argmax\{ z_i: z^{ij}=(z_i, z_j)^T\in \R^2, \tilde
  \gamma^{ij}(z^{ij})=0\}. 
\end{equation}
Therefore, we have $  \tilde\theta^{ij(i,\ray)}_j=\theta^{(i,\ray)}_j$
if and only if (\ref{eq:twodimthetamax}) holds. Thus, we have proved
that $ \tilde\theta^{ij(i,\ray)}$   is the symmetric point  
of $\bigl(\theta^{(i,\ray)}_i, \theta^{(i,\ray)}_j\bigr)^T$ on $\tilde \gamma^{ij}(z^{ij})=0$. Similarly, 
we can verify that $ \tilde\theta^{ij(j,\ray)}$   is the symmetric point  
of $\bigl(\theta^{(j,\ray)}_i, \theta^{(j,\ray)}_j\bigr)^T$   on
$\tilde \gamma^{ij}(z^{ij})=0$. Condition (\ref{eq:identicalSymmetry})
implies that 
\begin{displaymath}
  \tilde \theta^{ij(i,\ray)} = \tilde \theta^{ij(j,\ray)}.
\end{displaymath}
Thus,  Condition (5.2)
of \cite{DaiMiya2013} is satisfied for the  
$(\tilde
\Sigma^{ij}, \tilde \mu^{ij}, \tilde R^{ij})$-SRBM.

 It follows that
 the two-dimensional $(\tilde
\Sigma^{ij}, \tilde \mu^{ij}, \tilde R^{ij})$-SRBM has a product form
stationary distribution. Furthermore, it follows from (5.28) in the
proof of Theorem 5.1 in  \cite{DaiMiya2013},   
there exist two  constants $d^{(ij)}_i>0$ and $d^{(ij)}_j>0$ such that
for any $z^{ij}\in \R^2$
\begin{eqnarray*}
  \lefteqn{\tilde{\gamma}^{ij}(z_{i}, z_{j}) = d^{(ij)}_i \Delta_{i} c_{ij}^{-1} (\theta^{(j,\ray)}_{j} z_{i} - \theta^{(j,\ray)}_{i} z_{j}) (\alpha_{i} - z_{i})} \hspace{10ex} \nonumber\\
  && + d^{(ij)}_j \Delta_{j} c_{ij}^{-1}(-\theta^{(i,\ray)}_{j} z_{i} + \theta^{(i,\ray)}_{i} z_{j}) (\alpha_{j} - z_{j}),
\end{eqnarray*}
where $\alpha_{k} = \theta^{(k,\ray)}_{k} > 0$ for $k=i,j$. 
For $y^{ij}=(y_i, y_j)^T\in \R^2$, by letting $z^{ij} = A^{ij} y^{ij}$, we get
\begin{eqnarray}
\label{eq:characterization-1}
- \frac 12 \brb{ y_{ij},
{\Sigma^*}^{ij} y^{ij}} -
\brb{{\mu^*}^{ij},  y^{ij}}
   = \sum_{k=i,j}  d^{(ij)}_k\Delta_{k} y_{k} \left(\alpha_{k} - (\theta^{(i,\ray)}_{k} y_{i} + \theta^{(j,\ray)}_{k} y_{j}) \right).
\end{eqnarray}
Comparing the coefficients of $y_i$ on the both sides of
(\ref{eq:characterization-1}), we have 
\begin{displaymath}
 d^{(ij)}_i  \Delta_i \alpha_i = -\mu^*_i = - {\mu^*}^{ij}_i,
\end{displaymath}
from which we conclude that $d^{(ij)}_i$ is independent of $j$, and we
denote it by $d_i$.  Then \eq{characterization-1} becomes
\begin{eqnarray}
\label{eq:characterization-2}
- \frac 12 \brb{ y^{ij},
{\Sigma^*}^{ij} y^{ij}} -
\brb{{\mu^*}^{ij},  y^{ij}}
= \sum_{k=i,j}  d_k\Delta_{k} y_{k} \left(\alpha_{k} - (\theta^{(i,\ray)}_{k} y_{i} + \theta^{(j,\ray)}_{k} y_{j}) \right),
\end{eqnarray}
from which we have 
\begin{equation}
  \label{eq:Sigmastarexpre}
{\Sigma^*}_{ij} = \frac{1}{2}\Bigl (d_i \Delta_i \theta^{(j,\ray)}_i + d_j
\Delta_j \theta^{(i,\ray)}_j\Bigr) \quad  \text{and} \quad {\mu^*}_i=
- d_i \Delta_i \quad \text{ for any } i,j\in J.
\end{equation}
It follows from (\ref{eq:Sigmastarexpre}) that 
\begin{equation}
  \label{eq:characterizationstar}
- \frac 12 \brb{ y,
\Sigma^* } -
\brb{\mu^*,  y}
=\sum_{i=1}^d d_i \Delta_{i} y_{i} \left(\alpha_{i} - \sum_{k=1}^{d}
  \theta^{(k,\ray)}_{i} y_{k} \right) \quad \text{ for any } y \in \dd{R}^{d}.
\end{equation}
Setting $\theta=Ay$, we have
\begin{eqnarray*}
\gamma(\theta)=\sum_{i=1}^d{d_i\gamma_{i}(\theta)(\alpha_{i}-\theta_{i})}
\quad \text{ for any } \theta\in\R^d.
\end{eqnarray*}
Thus,
\eq{characterization} holds. It follows from \lem{characterization}
that the $d$-dimensional $(\Sigma, \mu, R)$-SRBM has a product form
stationary distribution.  This completes the proof of \thr{main}. 

\begin{proof}[Proof of Corollary \ref{cor:2}]
(a) The if and only if part is immediate from \cor{main} because we can see in the proof of \thr{main} that the conditions  \eq{tauijongamma}, \eq{thetaithetamax} and \eq{thetajthetamax} are equivalent for the two dimensional $(\tilde{\Sigma}^{ij}, \tilde{\mu}^{ij}, \tilde{R}^{ij})$-SRBM to have a product form stationary distribution. Thus, we only need to prove that, if the $d$-dimensional $(\Sigma, \mu, R)$-SRBM has a product form stationary distribution $Z(\infty)$, then
the density function of
$(Z_i(\infty), Z_j(\infty))$ is equal to  
\begin{displaymath}
\alpha_i\alpha_j
e^{-(\alpha_i y_i+\alpha_j y_j)}.  
\end{displaymath}
 In the proof of \thr{main}, we know
that when
(\ref{eq:characterization}) holds for every $\theta \in \R^d$, then
\eq{gammaijcharacterization} holds. Equation
\eq{gammaijcharacterization} is precisely the two-dimensional analog
of (\ref{eq:characterization}) for the two-dimensional
$(\tilde\Sigma^{ij}, \tilde \mu^{ij}, \tilde
R^{ij})$-SRBM. Therefore, by invoking
Lemma \ref{lem:characterization} again, this time in two dimensions,
we conclude that  
the stationary distribution
of the two-dimensional $(\tilde\Sigma^{ij}, \tilde \mu^{ij}, \tilde
R^{ij})$-SRBM is of product form with density 
$\alpha_i\alpha_j
e^{-(\alpha_i y_i+\alpha_j y_j)}$. This proves (a).  Then (b) is
immediate from (a) and \lem{pmatrix} because $R$ is assumed to be a
$\sr{P}$-matrix.
\end{proof}

\section{Tandem queues and variational problems} 
\label{sect:example}

In this section, we focus on SRBMs that arise from tandem queueing
networks. For such an SRBM, we characterize its product form
stationary distribution through its basic network parameters. We will
also discuss a variational problem (VP) associated with the SRBM.

We assume that the reflection
matrix $R$, the 
covariance matrix $\Sigma$, and the drift vector $\mu$ are given by
\begin{eqnarray}
  \label{eq:tandemRSigma}
 &&  R_{i, i-1} = -1 \quad \text{ and } \quad \Sigma_{i, i-1}= \Sigma_{i-1, i}=
 -c_{i-1}  \quad \text{ for } 
 i =2, \ldots, d, \\ 
 &&  R_{i,i}=1 \quad \text{ and } \quad \Sigma_{i,i} = c_{i-1} + c_i  \quad \text{ for }
 i =1, \ldots, d, \\
 &&   \mu_{i} = \beta_{i-1} - \beta_{i}  \quad \text{ for }  i =1, \ldots, d
\end{eqnarray}  
with all other entries being zero.
An example, when $d=3$, is 
 given by 
\begin{equation}
  \label{eq:Rmu}
  R=
  \begin{pmatrix}
    1 & 0 & 0 \\
   -1 & 1 & 0 \\
   0  & -1 & 1 
  \end{pmatrix}, 
\quad 
  \Sigma =
\begin{pmatrix}
  c_0+c_1 & -c_1 & 0 \\ 
 -c_1 & c_1+c_2  & -c_2\\
  0 & -c_2 & c_2+ c_3 
\end{pmatrix}, \quad 
\mu=
\begin{pmatrix}
  \beta_0-\beta_1 \\ \beta_1-\beta_2 \\ \beta_2-\beta_3
\end{pmatrix}.
\end{equation}
Such an SRBM arises from a $d$-station generalized Jackson network in
series, also known as a tandem queue. In the tandem queue, the
interarrival times are assumed to be iid with mean $1/\beta_{0}$ and squared coefficient of
variation (SCV) $c_0$. The service times at station $i$ are assumed
to be iid with mean $1/\beta_{i}$ and SCV $c_i$, $i\in J$. We assume  that
$\Sigma$ is nonsingular and  condition (\ref{eq:stability}) is satisfied.
It follows from \cite{HarrWill1987} that the SRBM $Z$ has a
unique stationary distribution $\pi$. By using \thr{main}, we first 
 check that the stationary distribution $\pi$
 has a product form  if and only if 
 \begin{equation}
   \label{eq:tandemProductForm}
   c_0=c_i \quad \text{ for } i = 1, \ldots, d-1.
 \end{equation}
To see this, set
\begin{equation}
  \label{eq:stable}
 b = -  R^{-1}\mu >0 \quad \text{ and } \quad
  \tau_i = \frac{2b_i}{c_0+c_i} \quad \text{ for } i =1,\ldots, d.
\end{equation}
We can easily compute that 
  \begin{equation}
    \label{eq:thetaitandm}
     \theta^{(i,\ray)}_{j} =  \tau_{i} \quad \text{ for } j =1, \ldots i \quad \text{ and }
     \quad \theta^{(i,\ray)}_{j}=0 \quad \text{ for } j = i+1, \ldots d.
   \end{equation}
Recall the definition of  $\Sigma^{*}$ and $\mu^{*}$ in (\ref{eq:Sigmastar}).
An easy computation leads to 
\begin{eqnarray*}
  \Sigma^{*}_{i, j}= c_{0} \tau_{i}\tau_{j} , \qquad
  \Sigma^{*}_{i,i} = \tau_{i}^{2}(c_{0} + c_{i}),\qquad
  \mu_i^*=\tau_i (\beta_0-\beta_i), \qquad   
 1\le i\ne j \le d.
\end{eqnarray*}  
Thus, for any $1 \le i<j \le d$, we have
\begin{eqnarray*}
 && \tilde{\Sigma}^{ij} = \left(\begin{array}{cc} c_{0}+c_{i} & -c_{i} \\
  -c_{i} & c_{i}+c_{j} \end{array}\right), \quad
 \tilde{\mu}^{ij} = \left(\begin{array}{cc} \beta_{0}-\beta_{i} \\
     \beta_{i}-\beta_{j} \end{array}\right), \quad
 \tilde{R}^{ij} = \left(\begin{array}{cc} 1 & 0 \\
  -1 & 1 \end{array}\right),
\end{eqnarray*}
where, in the derivation, we have used the following formula for
$A^{ij}$ defined in (\ref{eq:A})
\begin{eqnarray*}
  A^{ij} = \begin{pmatrix}
  \tau_{i} & \tau_{j} \\
  0 & \tau_{j} 
\end{pmatrix}.
\end{eqnarray*}
Because $\tilde R^{ij}$ is an ${\cal M}$-matrix,
 a $(\tilde{\Sigma}^{ij}, \tilde{\mu}^{ij}, 
\tilde{R}^{ij})$-SRBM is well defined and this 
two-dimensional SRBM  corresponds to a
two-station tandem queue consisting of station $i$ and station
$j$ from the original $d$-station network. So \cor{2} concludes that the SRBM from a
$d$-station tandem queue has the product form stationary distribution
if and only if each of the $\frac{1}{2}d(d-1)$ two-dimensional SRBMs
from the two-station queues have product form stationary
distributions. 

By solving equation (\ref{eq:quadraticzj}), we have
the symmetry points of $\theta^{(i,\ray)}$ and $\theta^{(j,\ray)}$:
\begin{eqnarray}
\label{eq:theta}
{\theta}^{ij(i,\ray)}=f^{ij}\left(\tau_{i} ,\,\,
     \frac{2c_{i}\tau_{i}+2b_{j}-2b_{i}}{c_{i}+c_{j}}\right) \quad
 \text{and}\quad
{\theta}^{ij(j,\ray)}=f^{ij}\left( \frac{ 2b_{i}+ c_{i}\tau_{j} -
    c_0\tau_{j}}{c_0+c_{i}},\,\,  \tau_{j}\right), 
\end{eqnarray}
where $f^{ij}$ is again the mapping defined in (\ref{eq:fij}).
Using \eq{theta}, condition (\ref{eq:identicalSymmetry}) is equivalent to
\begin{eqnarray*}
\tau_{i}=\frac{2b_{i}+ c_{i}\tau_{j}- c_0\tau_{j}}{c_0+c_{i}} \quad
\text{and} \quad
\tau_{j}=\frac{2c_{i}\tau_{i}+2b_{j}-2b_{i}}{c_{i}+c_{j}},
\end{eqnarray*}
which is further equivalent to (\ref{eq:tandemProductForm}).
Thus, we have used \thr{main} to prove that the $d$-dimensional stationary
distribution has a product form if and only if 
\begin{eqnarray}
\label{eq:tandem PF 1}
  c_0=c_{i} \quad \mbox{ for } i = 1,2, \ldots, d-1.
\end{eqnarray}
 This fact is well
   known and can be verified by using skew symmetry condition
   (\ref{eq:skew symmetric}) developed in \cite{HarrWill1987a}.

Recall the variational problem (VP) defined in Definition 2.3 of
\cite{AvraDaiHase2001}. The VP is proved to be related to 
large deviations rate function of the corresponding SRBM; see, for
example, \cite{Maje1996}. 
 In the two-dimensional case, the VP is 
solved completely in \cite{AvraDaiHase2001}, whose optimal solutions
are interpreted geometrically in \cite{DaiMiya2013}. In particular,
for the two-dimensional  $(\tilde{\Sigma}^{ij}, \tilde{\mu}^{ij},
\tilde{R}^{ij})$-SRBM,  the ``entrance'' velocities in (3.4) of
\cite{AvraDaiHase2001} 
are given by
\begin{eqnarray}
\label{eq:2d verocity}
\tilde a^{ij(i,\ray)} = \tilde \Sigma^{ij}
\tilde \theta^{ij(i,\ray)} + \tilde \mu^{ij} \quad \text{ and }  \quad
\tilde a^{ij(j,\ray)} = \tilde \Sigma^{ij}
\tilde \theta^{ij(j,\ray)} + \tilde \mu^{ij},
\end{eqnarray}
where $\tilde \theta^{ij(i,\ray)}$ is the two-dimensional vector whose 
components are the $i$th and $j$th component of $\theta^{ij(i,\ray)}$,
and $\tilde \theta^{ij(j,\ray)}$ is defined similarly.
These velocities indicate influence of the boundary faces on an optimal path, that is, a sample path for the optimal solution of the VP. See Section 4 of \cite{DaiMiya2013}. 

Assume the product form condition (\ref{eq:tandemProductForm}).  Under condition
(\ref{eq:tandemProductForm}), $\tau_i$ has the simplified
expression:
\begin{eqnarray}
\label{eqn:v pf}
  \tau_{i} = \frac 1{c_{0}} (\beta_{i} - \beta_{0}), \quad i=1,2,\ldots, d-1, \quad \mbox{ and } \quad \tau_{d} = \frac 2{c_{0}+c_{d}} (\beta_{d} - \beta_{0})
\end{eqnarray}
and the symmetry points are given by
\begin{eqnarray}
\label{eq:theta v}
 {\tilde \theta}^{ij(i,\ray)} = \tilde \theta^{ij(j,\ray)}
= (\tau_i, \tau_j)^T.  
\end{eqnarray}
Therefore, it follows from \eq{2d verocity} that, for $1 \le i < j \le d-1$,
\begin{eqnarray}
  \label{eq:ij velocity 1}
 \tilde{a}^{\{i,j\}} \equiv \tilde a^{ij(i,\ray)} =
\tilde a^{ij(j,\ray)} = c_{0}
\begin{pmatrix}
  \tau_{i} - \tau_{j} \\ 
  \tau_{j}
\end{pmatrix}
=
\begin{pmatrix}
\beta_i -\beta_j \\
 \beta_j  -\beta_0
\end{pmatrix}, 
\end{eqnarray}
  and, for $1 \le i \le d-1$,
\begin{eqnarray}
  \label{eq:ij velocity 2}
 \tilde{a}^{\{i,d\}} \equiv \tilde a^{id(i,\ray)} =
\tilde a^{id(d,\ray)} = 
\begin{pmatrix}
  c_{0}(\tau_{i} - \tau_{d}) \\  \displaystyle
  \frac 12 (c_{0} + c_{d}) \tau_{d}
\end{pmatrix}
=
\begin{pmatrix} \displaystyle
\beta_i - \frac {2c_{0} \beta_d + (c_{d} - c_{0}) \beta_{0})} {c_{0} + c_{d}} \\
 \beta_d  -\beta_0
\end{pmatrix}.
\end{eqnarray}

We now consider the optimal path for the VP for the product form network. To make arguments simplified, we consider the case for $d=3$ and assume that $c_{0} = c_{3}$ in addition to the product form condition \eq{tandem PF 1}. Then, we have \eq{ij velocity 1} for $1 \le i < j \le d = 3$.

Analogously to the two dimensional case in \cite{DaiMiya2013}, let us
consider a normal vector at a point $\theta$ on  the ellipse $E$.
 Denote this normal vector by $n^{J}(\theta)$. Then
it is easy to see that 
\begin{eqnarray*}
  n^{J}(\theta) = \Sigma \theta + \mu,
\end{eqnarray*}
which is denoted by $n^{\Gamma}(\theta)$ in (3.16) of
\cite{DaiMiya2013}. We conjecture $n^{J}(\tau)$ is the ``entrance
velocity'' for the last  segment of an optimal path from origin to a
point $z\in S$.  This conjecture is consistent with the result in the two
dimensional case; see Figure 3 of \cite{AvraDaiHase2001}.
Let $\tilde{a}^{J}=n^{J}(\tau)$. Then we have 
\begin{eqnarray*}
  \tilde{a}^{J} = n^{J}(\tau) = 
 \begin{pmatrix}
  \beta_{1} - \beta_{2} \\
  \beta_{2} - \beta_{3} \\
  \beta_{3} - \beta_{0}
\end{pmatrix}.
\end{eqnarray*}
Combining this $\tilde{a}^{J}$ with the two dimensional velocities
$\tilde{a}^{\{i,j\}}$, we can guess the optimal path for the
three-dimensional VP. 

To see this, let us consider the case that 
\begin{equation}
  \label{eq:3}
\beta_{1} < \beta_{2} <
\beta_{3}.  
\end{equation}
In this case, the first two components of $\tilde{a}^{J}$ are
negative, and the third component is positive. This suggests that the
final segment of the optimal path to a point $z \in S \equiv
\dd{R}_{+}^{3}$ with $z_3>0$ is parallel to $\tilde{a}^{J}$, and is a
straight-line 
from a point $y$ in the interior of the boundary face $F_{3} = \{x  
\in \dd{R}^{3}_+; x_{3} = 0\}$.  The optimal path from origin to $y$
should remain on face 
$F_{3}$ and is obtained using the velocity 
$\tilde{a}^{\{1,2\}}$ as argued in \cite{AvraDaiHase2001}. By \eq{ij
  velocity 1} and assumption (\ref{eq:3}), the first component of
$\tilde{a}^{\{1,2\}}$ 
is negative, and the second component is positive. Hence, the optimal path
in $F_{3}$ has two segments such that the first segment is from
the origin to a point on the first coordinate and the second segment
is from that point to $y$ by a straight-line.  

Thus, we conjecture that the optimal path from origin to $z$ is
composed of three segments whose first segment is on the first
coordinate axis, the second segment is from the end of the first
segment to $y \in F_{3}$, then the final segment is from $y$ to $z \in
S$. The optimality of this path is intuitively appealing because the
first queue is a bottleneck among the three queues and the second
queue is a bottleneck among the latter two queues under assumption
(\ref{eq:3}).

\vspace{5ex}

\appendix

\noindent {\bf \Large Appendix}

\section{The skew symmetric condition} \label{sec:skew}
We will use \lem{characterization} to show that SRBM has a product form stationary distribution of the form in
  (\ref{eq:density1}) if and only if 
 \eq{skew symmetric} holds.

For that, we show that \eq{characterization} holds with $C_{i}=\frac{\Sigma_{ii}}{2R_{ii}}$ for $1 \le i \le d$ and $\alpha$ given by \eq{alpha} is equivalent to \eq{skew symmetric}. As both sides of \eq{characterization}
are quadratic functions of $\theta\in \R^d$, \eq{characterization} holds if and only if
coefficients of  $\theta_{i}\theta_{j}$ and coefficients of
$\theta_{i}$ on the both sides are equal for all $i,j \in J$.  Letting
the coefficient of $\theta_{i}\theta_{j}$ of two sides equal, we
arrive at 
\begin{equation}
\label{eq:ssc}
-\frac{1}{2}(\Sigma_{ij}+\Sigma_{ji})=-C_{i}R_{ji}-C_{j}R_{ij}.
\end{equation}
for $1\le i,j\le d$. Let $j=i$ in \eq{ssc}, we can get $C_{i}=\frac{\Sigma_{ii}}{2R_{ii}}$. Let $i=j$ in \eq{ssc}, we can get $C_{j}=\frac{\Sigma_{jj}}{2R_{jj}}$. Rewrite \eq{ssc} into the matrix form with $C_{i}=\frac{\Sigma_{ii}}{2R_{ii}}$ and $C_{j}=\frac{\Sigma_{jj}}{2R_{jj}}$, we have \eq{skew symmetric}. Letting the coefficient of $\theta_{i}$ of two sides equal, we have
\begin{eqnarray*}
-\mu_{i}=\sum_{k=1}^d\frac{\Sigma_{kk}}{2R_{kk}}R_{ik}\alpha_k.
\end{eqnarray*}
We can also rewrite it into matrix form
\begin{eqnarray*}
\label{eq:mu}
-\mu = 2\diag(\Sigma)\diag(R)^{-1}R\alpha.
\end{eqnarray*}
Solving $\alpha$ from it, we can arrive at \eq{alpha}. 

So if \eq{characterization} holds, then $C_{i}=\frac{\Sigma_{ii}}{2R_{ii}}$ for $1 \le i \le d$ and $\alpha$ must be given by \eq{alpha}. And \eq{skew symmetric} holds. Conversely, if \eq{skew symmetric} holds, $C_{i}=\frac{\Sigma_{ii}}{2R_{ii}}$ for $1 \le i \le d$ and $\alpha$ given by \eq{alpha}, we have \eq{characterization} holds. Through \eq{characterization}, we have now reproduced a
result in \cite{HarrWill1987a}.

\section{Proofs of lemmas} \label{sec:proofs-lemmas}
\begin{proof}[Proof of Lemma~\ref{lem:hyper-2d}]

In the proof, $R^{|ij}$ is the $(d-2)\times
(d-2)$ principal submatrix of $R$ obtained by deleting rows $i$ and
$j$ and columns $i$ and $j$.

We first prove the existence of a map.
Let function $f^{ij}(z^{ij})$ be given in 
\eq{fij}. Clearly, $f^{ij}(z^{ij})\in \Gamma_{\{i,j\}}$ as it is  a linear combination of
$\theta^{(i,\ray)}$ and $\theta^{(j,\ray)}$. We first show that 
for any $\theta\in \Gamma_{\{i,j\}}$, we have $f^{ij}(\theta_i,
\theta_j)=\theta$. To see this, let
$\theta=a\theta^{(i,\ray)}+b\theta^{(j,\ray)}$ for some $a, b\in \R$. Then,
\begin{eqnarray*}
&&  \theta_i = a\theta^{(i,\ray)}_i+b\theta^{(j,\ray)}_i , \\
&&  \theta_j = a\theta^{(i,\ray)}_j+b\theta^{(j,\ray)}_j.
\end{eqnarray*}
Because $c_{ij}=\det(A^{ij})\neq 0$, we have
\begin{displaymath}
  \begin{pmatrix}
    a \\ b 
  \end{pmatrix} = \frac{1}{c_{ij}} 
  \begin{pmatrix}
    \theta^{(i,\ray)}_i \theta_i - \theta^{(j,\ray)}_i \theta_j \\
   -\theta^{(i,\ray)}_j \theta_i + \theta^{(j,\ray)}_j \theta_j 
  \end{pmatrix}.
\end{displaymath}
Using the definition of $f^{ij}$, it is clear that $f^{ij}(\theta_i,
\theta_j)=a\theta^{(i,\ray)}+b\theta^{(j,\ray)}=\theta$. 

To see the uniqueness of map $f^{ij}$, let 
$\theta^1$ and $\theta^2$ be two points on $\Gamma_{\{i,j\}}$. 
Assume that $\theta^1_{i}=\theta^2_{i}=z_i$,
$\theta^1_{j}=\theta^2_{j}=z_j$ for some $z_i$ and $z_j$. 
We now show that $\theta^1 = \theta^2$. To see this, 
let $\theta= \theta^1-
\theta^2$. Then $\langle R^{(k)}, \theta \rangle =0
$ for $k\in J\setminus\{i,j\}$ and $\theta_i=0$ and  $\theta_j=0$. It follows that 
\begin{equation}
  \label{eq:Rsubij}
  (R^{|ij})^{\rs{t}} \theta^{|ij}=0, 
\end{equation}
where $R^{|ij}$ is the $(d-2)\times (d-2)$ principal sub-matrix of $R$
obtained by deleting rows $i$ and $j$ and columns $i$ and $j$ from
$R$, and $\theta^{|ij}$ is the $(d-2)$-dimensional sub-vector
of $\theta$ by deleting components $i$ and $j$ from $
\theta$.  Later on, we will prove $R^{|ij}$ is non-singular. Hence,
(\ref{eq:Rsubij}) implies $\theta^{|ij}=0$, which, together
with $\theta_i=0$ and  $\theta_j=0$, implies $\theta=0$. Thus, we
have proved the claim.

To see the non-singularity of $R^{|ij}$, if not, then there exists $\beta^{|ij} \neq 0$ such that $R^{|ij} \beta^{|ij}=0$. Now let $w=\sum_{l \neq i,j} \beta^{|ij}_{l}R^{(l)}$, we see $w \neq 0$ as $\beta^{|ij} \neq 0$ and $R^{(l)}$ are linearly independent for $l \neq i,j$. On the other hand, we obtained $w_{l}=0$ for $l \neq i,j$ due to $R^{|ij} \beta^{|ij}=0$. Thus, $(w_i, w_j)^{\rs{t}} \neq 0$. Now considering $\langle w, \theta^{(i,\ray)}\rangle$, we get $\langle w, \theta^{(i,\ray)}\rangle=0$ as $\langle R^{(l)}, \theta^{(i,\ray)}\rangle=0$ for $l \neq i,j$. Then we obtain $w_i\theta^{(i,\ray)}_i+w_j\theta^{(i,\ray)}_j=0$. Similarly, we obtain $w_i\theta^{(j,\ray)}_i+w_j\theta^{(j,\ray)}_j=0$. So $A^{ij}$ is singular ($c_{ij}=0$) as $(w_i, w_j)^{\rs{t}} \neq 0$, a contradiction. Thus, we have proved $R^{|ij}$ is non-singular.
\end{proof}

\begin{proof}[Proof of \lem{pmatrix}]
(a) First we quote an equivalent definition of $\mathcal{P}$-matrix in
Section 2.5 of \cite{HornJohn1991}:
\emph{$A$ is a $\mathcal{P}$-matrix, if and only if for each nonzero $x\in \mathbb{R}^{d}$, there is some $k\in\{1,2,\cdots,d\}$ such that $x_{k}(Ax)_{k}>0$.}

For each nonzero $x\in \mathbb{R}^{d}$, $R^{-1}x$ is also nonzero. Since
$R$ is a $\mathcal{P}$-matrix, then there is some
$k\in\{1,2,\cdots,d\}$ such that $(RR^{-1}x)_k(R^{-1}x)_{k}>0$. Now we have
\begin{eqnarray*}
x_{k}(R^{-1}x)_{k}=(RR^{-1}x)_k(R^{-1}x)_{k}>0 .
\end{eqnarray*}
So $R^{-1}$ is also a $\mathcal{P}$-matrix.

(b) Assume that $R$ is a $\mathcal{P}$-matrix, then $R^{-1}$ is a
$\mathcal{P}$-matrix following (a). Thus, for $i\neq j$,
$\det((R^{-1})^{ij})>0$. On the other hand, using the fact that
$R^{-1}=\Delta^{-1} A^T$ we have
$\det((R^{-1})^{ij})=\frac{c_{ij}}{\Delta_{i}\Delta_{j}}$, where
$\Delta_i$ defined in (\ref{eq:Deltai}) is positive because of
(\ref{eq:stability}). Therefore, we have proved that $c_{ij}>0$.

(c) We next assume that the SRBM has a product form stationary distribution, and prove that $R$ is a $\sr{P}$-matrix. For that, we quote another equivalent definition of $\mathcal{P}$-matrix in \cite{HornJohn1991}:
\emph{$A$ is a $\mathcal{P}$-matrix, if and only if for each nonzero $x\in \mathbb{R}^{d}$, there is some positive diagonal matrix $D=D(x)\in M_d(\mathbb{R})$ such that $x^{\rs{T}}(D(x)A)x>0$.}

By Lemma~\ref{lem:characterization}, \eq{characterization}
holds.  Therefore, comparing the coefficients of $\theta_{i} \theta_{j}$, we have for any nonzero $\theta\in \R^d$, 
\begin{eqnarray}
\label{eq:positive definite}
\sum_{i=1}\sum_{j=1} C_i R_{ji} \theta_j \theta_i =
\frac 12 \langle\theta, \Sigma\theta\rangle>0,
\end{eqnarray}
where $C_i$'s are constants in (\ref{eq:characterization}) and the inequality holds because $\Sigma$ is positive definite. It
follows that 
\begin{displaymath}
x^T \diag(C_1, \ldots C_d)^{-1} R x >0
\end{displaymath}
for any nonzero $x\in \R^d$, proving that $R$ is a ${\cal P}$-matrix.
\end{proof}

\begin{proof}[Proof of Lemma \ref{lem:gammaijgamma}]
  One can check that 
\begin{eqnarray}
  \tilde{\gamma}^{ij}(z^{ij}) & = & 
- \frac 12 \brb{(A^{ij})^{-1} z^{ij} ,
  \tilde{\Sigma}^{ij} (A^{ij})^{-1}z^{ij}} -
\brb{\tilde{\mu}^{ij}, (A^{ij})^{-1}z^{ij}} \nonumber \\ 
 & = & 
- \frac 12 \brb{y,
{A^T\Sigma A} y } -
\brb{ A^T\mu, y} \nonumber\\
 & = & 
- \frac 12 \brb{Ay,
\Sigma (A y) } -
\brb{ \mu, Ay}  \nonumber\\
 & = & 
- \frac 12 \brb{f^{ij}(z^{ij}),
\Sigma f^{ij}(z^{ij}) }-
\brb{ \mu, f^{ij}(z^{ij})}  \nonumber  \\
&=& \gamma(f^{ij}(z^{ij})),\label{eq:gammaztheta}
\end{eqnarray}
where $y\in \R^d$ is the unique vector whose components $i$ and $j$
are given by 
$(A^{ij})^{-1} (z_i, z_j)^{\rs{t}}$ and other components are zero, and
in the second last equality we have used the fact that  $Ay=f^{ij}(z^{ij})$.
\end{proof}

\begin{proof}[Proof of lemma \ref{lem:uniqueSymmetry}]
According to Lemma \ref{lem:gammaijgamma}, \eq{quadraticzj} is equivalent to $\tilde \gamma(\theta^{(i,\ray)}_{i}, z_{j})=0$. Also due to Lemma \ref{lem:gammaijgamma}, \eq{extremeCondition} is equivalent to \eq{twodimthetamax}.

If \eq{twodimthetamax} holds, then for any solution $(\theta^{(i,\ray)}_{i}, z_{j})$ satisfying $\tilde \gamma(\theta^{(i,\ray)}_{i}, z_{j})=0$, we have $\frac{  \partial \tilde \gamma^{ij}(z_{ij})}{\partial z_j}|_{z^{ij}=(\theta^{(i,\ray)}_i,z_j)^{\rs{t}}}=-\tilde \Sigma_{jj}z_j-\tilde \Sigma_{ij}\theta^{(i,\ray)}_i-\tilde \mu_{j} =0$. So $\tilde \gamma(\theta^{(i,\ray)}_{i}, z_{j})=0$ has a unique solution $z_{j}=\theta^{(i,\ray)}_{j}$. Otherwise, $\frac{  \partial \tilde \gamma^{ij}(z_{ij})}{\partial z_j}
|_{z^{ij}=(\theta^{(i,\ray)}_i,\theta^{(i,\ray)}_j)^{\rs{t}}}=-\tilde \Sigma_{jj}\theta^{(i,\ray)}_j-\tilde \Sigma_{ij}\theta^{(i,\ray)}_i-\tilde \mu_{j} \neq 0$, so $\theta^{(i,\ray)}_{j} \neq -\frac{\tilde \Sigma_{ij}\theta^{(i,\ray)}_i+\tilde \mu_{j}}{\tilde \Sigma_{jj}}$. From quadratic equation $\tilde \gamma(\theta^{(i,\ray)}_{i}, z_{j})=0$, the other solution $z_j \neq \theta^{(i,\ray)}_{j}$ satisfies $z_{j}+\theta^{(i,\ray)}_{j} = -2\frac{\tilde \Sigma_{ij}\theta^{(i,\ray)}_i+\tilde \mu_{j}}{\tilde \Sigma_{jj}}$. So $z_{j} \neq \theta^{(i,\ray)}_{j}$.
\end{proof}

\section{Examples}
\label{sec:Examples}
Our first example complements Lemma~\ref{lem:pmatrix}.
\begin{example} {\rm
\label{exa:pmatrix}
Let
\begin{eqnarray}
  \label{eq:Rmu 2}
  R=
  \begin{pmatrix}
    1 & 1/2 & 1 & 0 \\
   2 & 1 & 0 & 1\\
   1  & 0 & 1 & 0\\
   0 & 1 & 0 &1 
  \end{pmatrix}.
\end{eqnarray}
Since $R$ is a nonnegative matrix, it is easy to  check that $R$
is  a complete-$\mathcal{S}$ matrix. The matrix 
$R$ in invertible with inverse
\begin{displaymath}
  R^{-1}=  \begin{pmatrix}
    0 & 1/2 & 0 & -1/2 \\
    2 & 0 & -2 & 0 \\
     0 & -1/2 & 1 & 1/2 \\
    -2 & 0 & 2 & 1
  \end{pmatrix}. 
\end{displaymath}
Then condition (\ref{eq:stability}) is satisfied with $\mu=-(1.1,
1.1, 1, 1)^T$. However, $c_{34}$ in (\ref{eq:cij}) equals zero,
demonstrating that 
  Lemma~\ref{lem:pmatrix} cannot be generalized to completely-${\cal
    S}$ matrix satisfying (\ref{eq:stability}).
}\end{example}

The next examples shows that, unlike the case when $d=2$, 
the condition that the point $\tau$, defined in (\ref{eq:1}), is on
the ellipse is not sufficient for a product form stationary
distribution.

\label{app:example}
\begin{example} {\rm
\label{exa:non-product form}
 Consider the 3-dimensional SRBM with
\begin{eqnarray}
  \label{eq:Rmu 1}
  R=
  \begin{pmatrix}
    1 & 0 & 0 \\
   -1 & 1 & 0 \\
   0  & -1 & 1 
  \end{pmatrix}, 
\quad 
  \Sigma =
\begin{pmatrix}
  1 & -1 & 0 \\ 
 -1 & 3  & -2\\
  0 & -2 & 3 
\end{pmatrix}, \quad 
\mu=
\begin{pmatrix}
  -\frac{1}{2} \\ -\frac{3}{2} \\ \frac{3}{2}
\end{pmatrix}.
\end{eqnarray}
Since $R$ is an ${\cal M}$-matrix, $R$ is completely-$S$.
One can verify that
\begin{eqnarray*}
 R^{-1}\mu=
  \begin{pmatrix}
   -\frac{1}{2} \\ -2 \\ - \frac {1}{2}
  \end{pmatrix}< 0.
\end{eqnarray*}

This SRBM arises from a three station tandem queue (see
\sectn{example} for details of this model). A simple computation leads
to $\theta^{(1,\ray)}=(1,0,0)^{\rs{t}}$,
$\theta^{(2,\ray)}=(2,2,0)^{\rs{t}}$ and
$\theta^{(3,\ray)}=(1,1,1)^{\rs{t}}$. Thus, $\tau=(1,2,1)^{\rs{t}}$, where
$\tau_{i}=\theta^{(i,\ray)}_{i}$  for $i=1, 2, 3$
following the definition in (\ref{eq:1})
. One can check that 
$\gamma(\tau) = 0$. Now we use Corollary
\ref{cor:main} to verify that  the SRBM does not have a product form
stationary distribution. For that,  we have
\begin{displaymath}
  \theta^{(12)}=f^{12}(\tau^{12})=
  \begin{pmatrix}
1 \\ 2 \\ 0
  \end{pmatrix}.
\end{displaymath}
Because $\gamma(\theta^{(12)})=-1\neq 0$, by Corollary \ref{cor:main}, 
this SRBM does not have a product form stationary distribution.
}\end{example}

\section{Equivalence of two versions of basic adjoint relationship}
\label{sec:equiv-two-vers}
This section is devoted to the proof for part (b) of Lemma
\ref{lem:key}. The key is to establish the equivalence of two versions
of basic adjoint relationship (BAR). This equivalence is stated in
Proposition \ref{pro:key 1} below.  Since this proposition may be of
independent interest, we keep this appendix as self-contained as
possible. This means 
that some of the terminology and notation are reintroduced here in
this appendix.

\subsection{The  main result}
\label{sect:main}
We focus on a $d$-dimensional semimartingale reflecting Brownian
motion (SRBM) that lives on the nonnegative orthant $\mathbb{R}^d_+$.
The SRBM data consists of a $d\times d$ positive definite matrix
$\Sigma$, a vector $\mu\in \R^d$ and a $d\times d$ reflection matrix
$R$.  The matrix $\Sigma$ is known as the covariance matrix, $\mu$ the
drift vector, and $R$ the reflection matrix. Assume that the SRBM has
a stationary distribution.  It is known that the stationary
distribution is unique and is characterized by a basic adjoint
relationship (BAR) (\cite{DaiKurt1994}). In this appendix, we show that a
moment generating function version of the BAR is equivalent to the
standard BAR in \cite{DaiHarr1992} and \cite{DaiKurt1994}. The
equivalence argument is standard. We present details here for easy
reference. 

Given the primitive data $(\Sigma, \mu, R)$ of an SRBM, we define the
following $d$-dimensional polynomials
\begin{eqnarray*}
 && \gamma(\theta)=-\frac{1}{2}\langle\theta, \Sigma\theta\rangle-\br{\mu, \theta}, \qquad \theta \in \dd{R}^{d},\\
 && \gamma_{i}(\theta) = \brb{R^{(i)}, \theta}, \qquad \theta \in
 \dd{R}^{d}, \quad i \in J=\{1, 2, \ldots, d\},
\end{eqnarray*}
where, for $x, y\in \R^d$, $\langle x, y\rangle$ denotes the standard
inner product of $x$ and $y$, and $R^{(i)}$ denotes the $i$th column
of $R$.
For a finite measure $\tau$  on $(\dd{R}_{+}^{d},
 \sr{B}(\dd{R}_{+}^{d})$ with $\sr{B}(\dd{R}_{+}^{d})$ being the Borel
 $\sigma$-field on $\dd{R}_{+}^{d}$, we define the corresponding
 moment generating function 
 \begin{displaymath}
\varphi_\tau(\theta) = \int_{\R^d_+} e^{\langle \theta,
  x\rangle}\tau(dx) \quad \text{ for } \theta\in \R^d \text{ with }
\theta\le 0.
 \end{displaymath}
Hereafter, vector inequalities are interpreted componentwise.
Because $\tau$ is a finite measure, $\varphi_\tau(\theta)$ is well
defined for each $\theta\le 0$. 
 When the measure $\tau$ is
clear from the context, we sometimes drop the subscript $\tau$ from
$\varphi_\tau$. For an open set $U\subset \R^m$ for some $m\ge 1$, 
a function $f:U\to \R$ is said to be in $C^k(U)$ if $f$ and its
derivatives up to $k$th order are continuous in $U$.
 A function $f:\R^d_+\to \R$ is said to be in
$C^2_b(\R^d_+)$ if (a) for each $x\in \R^d_+$, $f$ is well defined in a
neighborhood $U$ of $x$ in $\R^{d}$ such that $f\in C^2(U)$, and (b)
\begin{equation}
  \label{eq:fnorm}
\norm{f}_{\R^d_+}=\max_{i,j\in J} \sup_{x\in\R^d_+} \left
  \abs{\frac{\partial ^2 }{\partial 
    x_i \partial x_j}f(x)\right}
+ \max_{i\in J} \sup_{x \in\R^d_+} \left \abs{\frac{\partial }{\partial
    x_i}f(x)\right} + \sup_{x\in\R^d_+} \left \abs{f(x)\right}
\end{equation}
is finite.

\begin{proposition}
\label{pro:key 1}
Let $(\Sigma, \mu, R)$ be the data of an SRBM.
Assume that $\pi$ is a probability measure on $\dd{R}^d_+$ and that
$\nu_i$ is a positive finite measure whose support is contained in
$\{x\in\dd{R}^d_+: x_i=0\}$ for $i\in J$. Let $\varphi$ and $\varphi_i$
be the moment generating functions of $\pi$ and $\nu_i$,
respectively. Then $\varphi$, 
$\varphi_1$, $\ldots$, $\varphi_d$ satisfy 
\begin{eqnarray}
    \label{eq:key app1}
    \gamma(\theta) \varphi(\theta)= \sum_{i=1}^d \gamma_{i}(\theta)
    \varphi_{i}\bigl(\theta\bigr) \qquad \text{for each $\theta\in
\dd{R}^d$ with $\theta\le 0$}
\end{eqnarray}
if and only if 
\begin{equation}
\label{eq:bar}
\int_{\dd{R}^{d}_{+}} Lf(x) \pi(dx)+\sum_{i=1}^d\int_{\dd{R}^{d}_{+}
}D_i f(x)v_i(dx)=0 \qquad \text{for each $f \in C^2_b( \dd{R}^d_+)$,}
\end{equation}
where 
\begin{eqnarray*}
Lf(x)&=&\frac{1}{2}\sum_{i=1}^d\sum_{j=1}^d\Sigma_{ij}\frac{\partial^2 f}{\partial x_i \partial x_j}(x)+\sum_{i=1}^d\mu_i \frac{\partial f}{\partial x_i}(x),\\
D_if(x)&=&\sum_{j=1}^d R_{ji}\frac{\partial f}{\partial
  x_j}(x)\quad\text{for}\quad i\in J.
\end{eqnarray*}
\end{proposition}
\begin{remark}
Theorem 1.2 of \cite{DaiKurt1994} says if \eq{bar} holds for each $f
\in C^2_b( \dd{R}^d_+)$, then $\pi$ is the stationary distribution of
the SRBM
and $\nu_1$, $\ldots$, $\nu_d$ are the corresponding boundary
measures associated with the SRBM. Combining 
Theorem 1.2 of \cite{DaiKurt1994} with \pro{key 1}, we have proved
Lemma~\ref{lem:key}.

\end{remark}

\subsection{Proof of \pro{key 1}}
\label{sect:proof}

\begin{proof} We first argue that (\ref{eq:bar}) implies (\ref{eq:key
    1}).
Assume that (\ref{eq:bar}) holds for every $f\in C^2_b( \dd{R}^d_+)$.
For a given  $\theta\in \dd{R}^d$ with $\theta \le 0$, let 
\begin{equation}
  \label{eq:2}
  f(x)= e^{\langle \theta, x\rangle} \text{ for }
  x\in \dd{R}^d.
\end{equation}
One can verify that $f\in C^2_b( \dd{R}^d_+)$, $Lf(x)=\gamma(\theta)
e^{\langle \theta, x\rangle} $, and $D_if(x) =\gamma_i(\theta)
e^{\langle \theta, x\rangle}$. Since \eq{bar} holds for this $f$,
\eq{key 1} holds for this $\theta$.

Now we argue that  \eq{key 1} implies (\ref{eq:bar}).
Assume that \eq{key 1} holds.  We would like to prove that \eq{bar}
holds for each $f \in C^2_b( \dd{R}^d_+)$. In this section, we prove
this fact in four steps. Before we present full details of these
four steps,  we first provide an outline of these steps. 

Note that \eq{key 1} implies  \eq{bar} for all
functions $f$ of the form in (\ref{eq:2})
with $\theta\le 0$.
In step 1,
 we argue that \eq{bar} continues
to hold 
for functions $f$ of the form in (\ref{eq:2})  when $\theta$ is
replaced by $(z_1, \ldots, z_d)^T$,
where each $z_j$ is a complex variable with $\Re z_j\le 0$ and the
superscript $T$ 
represents transpose. In step 2,
applying the inverse Fourier theorem for any $f \in C^{\infty}_K(\R^d)$, the
space of
$C^{\infty}$ functions on $\R^d$ with compact support, we argue that \eq{bar} 
holds for each $f \in C^{\infty}_K(\R^d)$.  In step 3, we prove \eq{bar} holds 
for all $C^2_K(\R^d)$ functions.  In step 4, we prove \eq{bar} holds 
for all $C^2_b(\R^d_+)$ functions.   

Before we carry out the details of these four steps, we state a
standard result from complex analysis in the following lemma. The
lemma is used in step 1 below; for its proof, see, for example,
Theorem~1.1 on page 73 of \cite{SteiShak2010}.
\begin{lemma}
\label{lem:analytic}
Let $\Omega\subset\mathbb{C}$ be some connected open subset of the
complex plane $\dd{C}$ and let $f$ be an analytic function defined on
$\Omega$. Suppose that $f(z_0)=0$ for some $z_0\in \Omega$. Then, 
either $f(z)=0$ for all $z\in \Omega$ or there exists a neighborhood
$U\subset \Omega$ of $z_0$ such that $f(z)\neq 0$ for all $z\in
U\setminus \{z_0\}$.
\end{lemma}

{\bf Step 1.}  Let $z=(z_1, \ldots, z_d)^T$ and each $z_j$ be a
complex variable with $\Re z_j\le 0$. Let $f(x)=e^{\langle z,
  x\rangle}$. Define
\begin{displaymath}
h(z) =\int_{\dd{R}^{d}_{+}} Lf(x)
\pi(dx)+\sum_{i=1}^d\int_{\dd{R}^{d}_{+} }D_i f(x)v_i(dx).
\end{displaymath}
Since $Lf(x)=-\gamma(z)f(x)$, $D_if(x)=\gamma_i(z)f(x)$, and $\pi$ and
$\nu_i$ are finite measures, 
one can check that $h(z)$ is well defined and it satisfies
\begin{equation}
    h(z)= -\gamma(z) \varphi(z) + \sum_{i=1}^d \gamma_i(z)
  \varphi_i(z). \label{eq:h1}
\end{equation}

First, we would like to prove $h(z)=0$, where $z=(z_1, \ldots, z_d)^T$ 
and each $z_j$ is a complex variable with $\Re z_j< 0$. 
To see this,  fix
$\theta_j<0$ for $2 \le j \le d$. Let  $U=\{z_1\in \mathbb{C}: \Re
z_1<0\}$. For any $z_1\in U$, define 
\begin{eqnarray*}
g_1(z_1)=h(z_1, \theta_2, \cdots, \theta_d),
\end{eqnarray*}
which, by (\ref{eq:h1}), is equal to 
\begin{eqnarray*}
 -\gamma(z_1, \theta_2, \cdots, \theta_d) \varphi(z_1, \theta_2,
\cdots, \theta_d) + \sum_{i=1}^d \gamma_i(z_1, \theta_2, \cdots, \theta_d)\varphi_i(z_1, \theta_2, \cdots, \theta_d), 
\end{eqnarray*}
where the argument inside functions such as $h(\cdot)$ should have
been the column vector 
$(z_1, \theta_2, \cdots, \theta_d)^T$; for notational simplicity, we
drop the transpose and write  $h(z_1, \theta_2, 
\cdots, \theta_d)$ in the rest of this document.
Clearly, $\gamma(z_1, \theta_2, \cdots, \theta_d)$ 
and  $\gamma_i(z_1, \theta_2, \cdots, \theta_d)$ are analytical
functions of $z_1$ in the entire complex plane $\dd{C}$. Also, one can
check that $\varphi(z_1, \theta_2, \cdots, \theta_d)$ and
$\varphi_i(z_1, \theta_2, \cdots, \theta_d)$ are analytical functions
of $z_1$ on $U$.  From (\ref{eq:key 1}), we know that
$h(\theta)=0$ for $\theta \in \dd R^d$ with $\theta<0$. Therefore,
$g_1(\theta_1)=0$ for $\theta_1\in (-\infty, 0)\subset U$. Applying
Lemma~\ref{lem:analytic}, we have $g_1(z_1)=0$ for $z_1\in U$.
Similarly, by fixing $z_1\in U$, $\theta_3<0$, $\ldots$, $\theta_d<0$, we
can prove $g_2(z_2)=h(z_1, z_2, \theta_3, \ldots, \theta_d)$ is an
analytic function on $U$ and $g_2(\theta_2)=0$ for
$\theta_2\in(-\infty,0)\subset U$. Therefore, again by Lemma
\ref{lem:analytic},  $g_2(z_2)=0$ for $z_2\in U$.
Thus, we have proved that for any $z_i\in \dd{C}$ with $\Re z_i<0$ for
$i=1,2$ and $\theta_i\in(-\infty, 0)$ for $i=3, \ldots, d$,
$h(z_1,z_2,\theta_3, \ldots,\theta_d)=0$. By an induction argument, 
one can prove that $h(z_1, \ldots, z_d)=0$ for $z_i\in\dd{C}$ with
$\Re z_i<0$ for $i\in J$.

Next, we would like to  prove $h(z)=0$, where $z=(z_1, \ldots, z_d)^T$ 
and each $z_j$ is a complex variable with $\Re z_j=0$. We use
an induction argument to prove this. Suppose that $h(z_1, \ldots,
z_{i-1}, z_i, z_{i+1}, \ldots z_d)=0$ for $z_j\in \dd{C}$ with $\Re
z_j=0$ for $j=1, \ldots, i-1$ and $z_j\in \dd{C}$ with $\Re z_j<0$ for
$j=i, \ldots, d$. Fix 
\begin{displaymath}
z=(z_1, \ldots z_{i-1}, z_i, z_{i+1}, \ldots, z_d)^T, 
\end{displaymath}
where $z_j \in \dd{C}$ for $j\in J$ with $\Re z_j=0$ for $j=1, \ldots, i$
and $\Re z_j<0$ for $j=i+1, \ldots, d$. 
For each positive integer $k$, let 
$z^k_i=z_i-1/k$. Then $\Re z^k_i<0$ for each $k\ge 1$ and $\lim_{k\to
  \infty} z^k_i=z_i$.
Let 
\begin{displaymath}
z^k=(z_1, \ldots, z_{i-1}, z^k_i, z_{i+1}, \ldots, z_d)^T \quad \text{
for } k\ge 1.
\end{displaymath}
Then $\lim_{k\to\infty} z^k =z$. 
Clearly, 
\begin{eqnarray*}
\lim_{k\to\infty}  \gamma(z^k)  =\gamma(z) \quad \text{and} \quad
 \lim_{k\to\infty}  \gamma_j(z^k)  =\gamma_j(z)\quad
\text{ for } j\in J.
\end{eqnarray*}
Note that 
\begin{eqnarray*}
  \varphi(z^k) 
= \int_{\R^d_+} f_k(x) \pi(dx),
\end{eqnarray*}
where  $f_{k}(x)=e^{\langle z^k, x\rangle}$. Since $\Re z^k_j\le 0$
for each $j\in J$,  one can check that  $\abs{f_k(x)}\le
1$ for each $x\in \R^d_+$. By the dominated convergence theorem, we have
\begin{displaymath}
\lim_{k\to\infty}\int_{\R^d_+} f_k(x) \pi(dx)  = \int_{\R^d_+} f(x) \pi(dx),
\end{displaymath}
where $f(x)$ is given in (\ref{eq:2}). Therefore, we have proved that 
\begin{displaymath}
 \lim_{k\to\infty}   \varphi(z^k)  =   \varphi(z).
\end{displaymath}
Similarly, we can prove 
\begin{displaymath}
 \lim_{k\to\infty}   \varphi_j(z^k)  =   \varphi_j(z)  \quad \text{
   for each } j\in J. 
\end{displaymath}
 By (\ref{eq:h1}), 
\begin{eqnarray*}
&& \lim_{k\to\infty } h(z^k)  = \lim_{k\to\infty } \Bigl(
  -\gamma(z^k) \varphi(z^k) + \sum_{j=1}^d \gamma_j(z^k)\varphi_j(z^k)
  \Bigr)\\
&=&
  -\gamma(z) \varphi(z) + \sum_{j=1}^d \gamma_j(z)\varphi_j(z) 
= h(z).
\end{eqnarray*}
 By the induction assumption, $h(z^k)=0$ for $k\ge 1$. Therefore, we
 have $h(z)=0$.
Thus, we have proved that  \eq{bar} holds for functions
$f(x)=e^{\langle z, x\rangle}$,
where $z=(z_1, \ldots, z_d)^T$
and each $z_i$ is a complex variable with $\Re z_i=0$.
 
{\bf Step 2.} In this step, we prove \eq{bar} holds for any function
$f \in C^{\infty}_K(\R^d)$, the space of $C^{\infty}$ functions on
$\R^d$ with compact 
support. Such an $f$ belongs to the so called Schwartz space on
$\dd{R}^d$ (page 236 in \cite{Foll1999}). For a function $f$ in the Schwartz space, its
Fourier transform $\hat{f}(\zeta)=\int_{\dd{R}^d}e^{-2\pi \imath
  \langle y, \zeta\rangle}f(y) dy$ is well defined for each $\zeta \in
\R^d$.  By the Fourier inversion theorem for the functions in Schwartz
space (Corollary 8.23 and Theorem 8.26 in \cite{Foll1999}),
 one can recover a function $f$ in Schwartz space through its Fourier transform
\begin{equation*}
f(x)=\int_{\dd{R}^d}e^{2\pi \imath  \langle x, \zeta\rangle}
\hat{f}(\zeta)d\zeta \quad \text{ for } x\in \R^d.
\end{equation*}

For any function $f$ belonging to Schwartz space, its Fourier
transform and itself are both absolutely integrable (Corollary 8.23 in
\cite{Foll1999}). Consequently,  we have following expressions for $L
f(x)$ and $D_i 
f(x)$ for $i\in J$ 
\begin{eqnarray*}
Lf(x)=\int_{\dd{R}^d} Lg(x,\zeta)\hat{f}(\zeta)d\zeta \quad \text{and} \quad
D_i f(x)=\int_{\dd{R}^d}D_i g(x,\zeta) \hat{f}(\zeta)d\zeta \quad
\text{ for } x\in \R^d,
\end{eqnarray*}
where $g(x,\zeta)=e^{2\pi \imath  \langle x, \zeta\rangle}$. Then one can check,
\begin{eqnarray*}
&&\int_{\dd{R}^{d}_{+}} Lf(x) \pi(dx)+\sum_{i=1}^d\int_{\dd{R}^{d}_{+} }D_i f(x)v_i(dx)\\
&&=\int_{\dd{R}^{d}_{+}} \int_{\dd{R}^d}Lg(x,\zeta) \hat{f}(\zeta)d\zeta \pi(dx)+\sum_{i=1}^d\int_{\dd{R}^{d}_{+} }\int_{\dd{R}^d}D_i g(x,\zeta)\hat{f}(\zeta)d\zeta v_i(dx)\\
&&=\int_{\dd{R}^d}\left\{\int_{\dd{R}^{d}_{+}} Lg(x,\zeta) \pi(dx)+\sum_{i=1}^d\int_{\dd{R}^{d}_{+} }D_i g(x,\zeta)v_i(dx)\right\}\hat{f}(\zeta)d\zeta\\
&&=0.
\end{eqnarray*}
The second equality is due to Fubini's Theorem. Fubini's Theorem holds
because $\hat{f}(\zeta)$ is absolutely integrable over $\dd{R}^d$. The
last equality holds because 
\begin{displaymath}
\int_{\dd{R}^{d}_{+}} Lg(x,\zeta)
\pi(dx)+\sum_{i=1}^d\int_{\dd{R}^{d}_{+} }D_i g(x,\zeta)v_i(dx)=0
\end{displaymath} for
all $\zeta \in {\dd R}^d$ and  the result in Step 2. Therefore we have
prove that 
\eq{bar} holds for 
$C^{\infty}$ functions with compact support. 

{\bf Step 3.} In this step, we prove \eq{bar} holds for all $C^2_K(\R^d)$
functions. Fix an $f(x) \in C^2_K(\R^d)$. We now construct a sequence
of functions $g^n(x)\in C^\infty_K(\R^d)$ that converges to $f$ in a
proper sense. The construction is standard and is adapted from
Proposition 8 on page 29 
of \cite{Yosi1980}.
 Let
$g^{n}(x)=\eta^{n}*f(x)=\int_{\R^d} \eta^n(y)f(x-y)dy$, where
$\eta^{n}(x)=n^d\eta(nx)$, 
\begin{equation}
  \label{eq:eta}
   \eta(x) = 
  \begin{cases}
    c\exp{(-{(1-\abs{x}^2)}^{-1})} & \text{ for } \abs{x}<1, \\
    0    & \text{ otherwise},
  \end{cases}
\end{equation}
and $c$ is a constant such that $\int_{\dd{R}^d}
\eta(x)dx=1$. 
It is known that  $\eta^n(x) \in 
C^{\infty}_K(\R^d)$ and  $g^{n}(x) \in C^{\infty}_K(\R^d)$.
By the result from Step 2,  we have
\begin{equation}
  \label{eq:gn}
\int_{{\dd R}^d_{+}} Lg^n(x) \pi(dx)+\sum_{i=1}^d\int_{{\dd
    R}^d_{+}}D_i g^n(x)v_i(dx)=0 \quad \text{ for each }  n\ge 1.
\end{equation}
Because $f\in C^2_K(\R^d)$, we have for each $x\in \R^d$
\begin{equation}
  \label{eq:Lgn}
 L g^{n}(x) =\eta^n*Lf(x) \quad \text{ and }  \quad D_ig^n(x) =
 \eta^n*D_if(x),
\end{equation}
and 
\begin{displaymath}
  \lim_{n\to\infty} Lg^n(x) = Lf(x) \quad\text{and}\quad 
  \lim_{n\to\infty} D_ig^n(x) = D_if(x), i\in J.
\end{displaymath}
Also, by (\ref{eq:Lgn}), one has for each $n\ge 1$
\begin{displaymath}
  \sup_{x\in\R^d} \abs{L g^{n}(x)} \le   \sup_{x \in\R^d} \abs{L f(x)} 
\quad \text{ and } \quad   \sup_{x\in\R^d} \abs{D_i g^{n}(x)} \le   \sup_{x\in\R^d}
\abs{D_i f(x)}.
\end{displaymath}
Taking $n\to\infty$ on both sides of (\ref{eq:gn}), by the bounded
convergence theorem, we have 
\begin{equation}
  \label{eq:barf}
\int_{{\dd R}^d_{+}} Lf(x) \pi(dx)+\sum_{i=1}^d\int_{{\dd
    R}^d_{+}}D_i f(x)v_i(dx)=0.
\end{equation}

{\bf Step 4.} In this step, we  prove that \eq{bar} holds for 
$f\in C^2_b(\R^d_+)$. 
 Fix an $f(x) \in C^2_b(\R^d_+)$. Then $\norm{f}_{\R^d_+}<\infty$, 
where $\norm{f}_{\R^d_+}$ is defined in (\ref{eq:fnorm}).
For any $ \epsilon >
0$, choose a constant $B>0$ such that
\begin{equation}
  \label{eq:epsilon}
\pi(\{|x|\ge B\})+\sum_{i=1}^d \nu_i(\{|x|\ge B\})<\epsilon.   
\end{equation}
Since $f\in C^2(\R^d_+)$, there exists $\delta>0$ such that $f$, its
first order derivatives, and second order derivatives are well defined 
and continuous on 
\begin{displaymath}
  \{x\in \R^d: x_i > -4\delta \text{ for } i\in J\}\cap \{\abs{x}< B+2\}.
\end{displaymath}
Let 
\begin{eqnarray*}
&&  h_1(y) = 1-\int_0^y \eta\bigl(u-(B+1)\bigr) du \quad \text{ for }
  y\in\R.   \\
&&  h_2(y) = \int_{-1}^{(y+2\delta)/\delta} \eta\bigl(u\bigr) du \quad
\text{ for } y\in \R, 
\end{eqnarray*}
where $\eta:\R\to \R$ is the one variable version defined in (\ref{eq:eta}).
One can check  that $h_k\in C^2(\R)$, $k=1,2$, 
\begin{displaymath}
  h_1(y) =
  \begin{cases}
    1 &  \text{ for } -\infty< y \le B, \\
    0 &  \text{ for } y \ge B+2,
  \end{cases}
\quad \text{and} \quad 
  h_2(y) =
  \begin{cases}
    0 &  \text{ for } -\infty< y \le -3\delta, \\
    1 &  \text{ for } y \ge -\delta.
  \end{cases}
\end{displaymath} 
Define 
\begin{displaymath}
  g(x)=
  \begin{cases}
    f(x)h_1(\abs{x})\prod_{i=1}^d h_2(x_i) & \text{ for } x\in \{x\in
    \R^d: x_i>-4\delta, i\in J\}, \\
    0 & \text{ otherwise},
  \end{cases}
\end{displaymath}
where $\abs{x}=\sqrt{\langle x, x\rangle}$.
Since $\abs{x}$ has
derivatives in all orders for $\abs{x}>B$, one can verify that $g \in
C^2_K(\R^d)$. 
It follows from Step 3 that 
\begin{equation}
  \label{eq:barg}
  \int_{\R^d_+} Lg(x)  \pi(dx)+\sum_{i=1}^d\int_{\R^d_+}D_i
  g(x)v_i(dx)=0.
\end{equation}
Because
\begin{displaymath}
  \label{eq:hbounds}
 \sup_{y\in \R}\abs{h_1(y)}\le 1, \quad   \sup_{y\in \R}\abs{h_1'(y)}\le
 1, \quad \text{and } \sup_{y\in \R}\abs{h_1''(y)}\le 3/2, 
\end{displaymath}
there exists a constant $C>0$, independent of $B$ and $\delta$, such that 
\begin{equation}
  \label{eq:normfinequality}
 \sup_{x\in\R^d_+} (\abs{Lf(x)} + \abs{Lg(x)}) \le C \norm{f}_{\R^d_+},
 \quad\sup_{x\in\R^d_+ }
(\abs{D_if(x)} + \abs{D_ig(x)})
\le C \norm{f}_{\R^d_+}.
\end{equation}

Note that for  $x\in \R^d_+\cap\{\abs{x}<B\}$, $Lf(x)=Lg(x)$ and
$D_if(x)=D_ig(x)$. Therefore, we have 
\begin{eqnarray*}
\lefteqn{\left\abs{\int_{\R^d_+} Lf(x)
    \pi(dx)+\sum_{i=1}^d\int_{\R^d_+}D_i f(x)v_i(dx)\right}} \\
&= &
\left\abs{\int_{\R^d_+\cap\{\abs{x}<B\}} Lg(x)
    \pi(dx)+\sum_{i=1}^d\int_{\R^d_+\cap\{\abs{x}<B\}}D_i g (x)v_i(dx)\right} \\
&& { }+
\left\abs{\int_{\R^d_+\cap\{\abs{x}\ge B\}} Lf(x)
    \pi(dx)+\sum_{i=1}^d\int_{\R^d_+\cap\{\abs{x}\ge B\}}D_i f(x)v_i(dx)\right} \\
&\le &
\left\abs{\int_{\R^d_+} Lg(x)
    \pi(dx)+\sum_{i=1}^d\int_{\R^d_+}D_i g (x)v_i(dx)\right} \\
&& { }+
\left\abs{\int_{\R^d_+\cap\{\abs{x}\ge B\}} Lg(x)
    \pi(dx)+\sum_{i=1}^d\int_{\R^d_+\cap\{\abs{x}\ge B\}}D_i g(x)v_i(dx)\right} \\
&& { }+
\left\abs{\int_{\R^d_+\cap\{\abs{x}\ge B\}} Lf(x)
    \pi(dx)+\sum_{i=1}^d\int_{\R^d_+\cap\{\abs{x}\ge B\}}D_i f(x)v_i(dx)\right} \\
&\le & C\norm{f}_{\R^d_+} \epsilon,
\end{eqnarray*}
where the last inequality follows from (\ref{eq:barg}),
(\ref{eq:normfinequality}), and (\ref{eq:epsilon}). Since $\epsilon>0$
can be arbitrarily small, we have that (\ref{eq:bar}) holds for $f\in
C^2_b(\R^d_+)$.

\end{proof}

\bibliography{dai05082014}

\def\cprime{$'$} \def\cprime{$'$} \def\cprime{$'$} \def\cprime{$'$}
  \def\cprime{$'$} \def\cprime{$'$} \def\cprime{$'$}
\begin{thebibliography}{25}
\expandafter\ifx\csname natexlab\endcsname\relax\def\natexlab#1{#1}\fi
\expandafter\ifx\csname url\endcsname\relax
  \def\url#1{\texttt{#1}}\fi
\expandafter\ifx\csname urlprefix\endcsname\relax\def\urlprefix{URL }\fi
\providecommand{\eprint}[2][]{\url{#2}}

\bibitem[{Avram et~al.(2001)Avram, Dai and Hasenbein}]{AvraDaiHase2001}
\textsc{Avram, F.}, \textsc{Dai, J.~G.} and \textsc{Hasenbein, J.~J.} (2001).
\newblock Explicit solutions for variational problems in the quadrant.
\newblock \textit{Queueing Systems}, \textbf{37} 259--289.

\bibitem[{Berman and Plemmons(1979)}]{BermPlem1979}
\textsc{Berman, A.} and \textsc{Plemmons, R.~J.} (1979).
\newblock \textit{Nonnegative matrices in the mathematical sciences}.
\newblock Academic Press, New York.

\bibitem[{Bramson et~al.(2010)Bramson, Dai and Harrison}]{BramDaiHarr2010}
\textsc{Bramson, M.}, \textsc{Dai, J.~G.} and \textsc{Harrison, J.~M.} (2010).
\newblock Positive recurrence of reflecting {Brownian} motion in three
  dimensions.
\newblock \textit{Annals of Applied Probability}, \textbf{20} 753--783.

\bibitem[{Dai and Harrison(1992)}]{DaiHarr1992}
\textsc{Dai, J.~G.} and \textsc{Harrison, J.~M.} (1992).
\newblock Reflected {B}rownian motion in an orthant: numerical methods for
  steady-state analysis.
\newblock \textit{Annals of Applied Probability}, \textbf{2} 65--86.

\bibitem[{Dai and Kurtz(1994)}]{DaiKurt1994}
\textsc{Dai, J.~G.} and \textsc{Kurtz, T.~G.} (1994).
\newblock Characterization of the stationary distribution for a semimartingale
  reflecting {Brownian} motion in a convex polyhedron.
\newblock Preprint.

\bibitem[{Dai and Miyazawa(2011)}]{DaiMiya2011}
\textsc{Dai, J.~G.} and \textsc{Miyazawa, M.} (2011).
\newblock Reflecting {Brownian} motion in two dimensions: Exact asymptotics for
  the stationary distribution.
\newblock \textit{Stochastic Systems}, \textbf{1} 146--208.

\bibitem[{Dai and Miyazawa(2013)}]{DaiMiya2013}
\textsc{Dai, J.~G.} and \textsc{Miyazawa, M.} (2013).
\newblock Stationary distribution of a two-dimensional {SRBM}: geometric views
  and boundary measures.
\newblock \textit{Queueing Systems}, \textbf{74} 181--217.

\bibitem[{Dai et~al.(2013)Dai, Miyazawa and Wu}]{DaiMiyaWu2013}
\textsc{Dai, J.~G.}, \textsc{Miyazawa, M.} and \textsc{Wu, J.} (2013).
\newblock Decomposable stationary distribution of a multidimensional {SRBM}.
\newblock Submitted for publication.

\bibitem[{Dai and Williams(1995)}]{DaiWill1995}
\textsc{Dai, J.~G.} and \textsc{Williams, R.~J.} (1995).
\newblock Existence and uniqueness of {Semimartingale} reflecting {Brownian}
  motions in convex polyhedrons.
\newblock \textit{Theory of Probability and Its Applications}, \textbf{40}
  1--40.
\newblock Correctional note: 2006, 59, 346--347.

\bibitem[{Folland and Folland(1999)}]{Foll1999}
\textsc{Folland, G.~B.} and \textsc{Folland, G.} (1999).
\newblock \textit{Real analysis: modern techniques and their applications},
  vol. 361.
\newblock Wiley New York.

\bibitem[{Harrison and Hasenbein(2009)}]{HarrHase2009}
\textsc{Harrison, J.~M.} and \textsc{Hasenbein, J.~J.} (2009).
\newblock Reflected {B}rownian motion in the quadrant: Tail behavior of the
  stationary distribution.
\newblock \textit{Queueing Systems}, \textbf{61} 113--138.

\bibitem[{Harrison and Nguyen(1993)}]{HarrNguy1993}
\textsc{Harrison, J.~M.} and \textsc{Nguyen, V.} (1993).
\newblock {Brownian} models of multiclass queueing networks: Current status and
  open problems.
\newblock \textit{Queueing Systems: Theory and Applications}, \textbf{13}
  5--40.

\bibitem[{Harrison and Williams(1987{\natexlab{a}})}]{HarrWill1987}
\textsc{Harrison, J.~M.} and \textsc{Williams, R.~J.} (1987{\natexlab{a}}).
\newblock {Brownian} models of open queueing networks with homogeneous customer
  populations.
\newblock \textit{Stochastics}, \textbf{22} 77--115.

\bibitem[{Harrison and Williams(1987{\natexlab{b}})}]{HarrWill1987a}
\textsc{Harrison, J.~M.} and \textsc{Williams, R.~J.} (1987{\natexlab{b}}).
\newblock Multidimensional reflected {B}rownian motions having exponential
  stationary distributions.
\newblock \textit{Annals of Probability}, \textbf{15} 115--137.

\bibitem[{Horn and Johnson(1991)}]{HornJohn1991}
\textsc{Horn, R.~A.} and \textsc{Johnson, C.~R.} (1991).
\newblock \textit{Topics in Matrix Analysis}.
\newblock Cambridge University Press, New York.

\bibitem[{Kang et~al.(2009)Kang, Kelly, Lee and Williams}]{KangKellLeeWill2009}
\textsc{Kang, W.}, \textsc{Kelly, F.}, \textsc{Lee, N.} and \textsc{Williams,
  R.} (2009).
\newblock State space collapse and diffusion approximation for a network
  operating under a fair bandwidth sharing policy.
\newblock \textit{The Annals of Applied Probability}, \textbf{19} 1719--1780.

\bibitem[{Kharroubi et~al.(2012)Kharroubi, Yaacoubi, Tahar and
  Bichard}]{KharYaacTahaBich2012}
\textsc{Kharroubi, A.~E.}, \textsc{Yaacoubi, A.}, \textsc{Tahar, A.~B.} and
  \textsc{Bichard, K.} (2012).
\newblock Variational problem in the non-negative orthant of $\mathcal{R}^{3}$:
  reflective faces and boundary influence cones.
\newblock \textit{Queueing Systems}, \textbf{70} 299--337.

\bibitem[{Latouche and Miyazawa(2014)}]{LatoMiya2014}
\textsc{Latouche, G.} and \textsc{Miyazawa, M.} (2014).
\newblock Product-form characterization for a two-dimensional reflecting random
  walk.
\newblock \textit{Queueing Systems}.
\newblock To appear.

\bibitem[{Liang and Hasenbein(2013)}]{LianHase2013}
\textsc{Liang, Z.} and \textsc{Hasenbein, J.~J.} (2013).
\newblock Optimal paths in large deviations of symmetric reflected {Brownian}
  motion in the octant.
\newblock \textit{Stochastic Systems}, \textbf{3} 187--229.

\bibitem[{Majewski(1996)}]{Maje1996}
\textsc{Majewski, K.} (1996).
\newblock Large deviations of stationary reflected {Brownian} motions.
\newblock In \textit{Stochastic Networks: Theory and Applications} (F.~P.
  Kelly, S.~Zachary and I.~Ziedins, eds.). Oxford University Press.

\bibitem[{Majewski(1998{\natexlab{a}})}]{Maje1998}
\textsc{Majewski, K.} (1998{\natexlab{a}}).
\newblock Heavy traffic approximations of large deviations of feedforward
  queueing networks.
\newblock \textit{Queueing Systems}, \textbf{28} 125--155.

\bibitem[{Majewski(1998{\natexlab{b}})}]{Maje1998a}
\textsc{Majewski, K.} (1998{\natexlab{b}}).
\newblock Large deviations of the steady-state distribution of reflected
  processes with applications to queueing systems.
\newblock \textit{Queueing Systems}, \textbf{29} 351--381.

\bibitem[{Reiman and Williams(1988)}]{ReimWill1988}
\textsc{Reiman, M.~I.} and \textsc{Williams, R.~J.} (1988).
\newblock A boundary property of semimartingale reflecting {B}rownian motions.
\newblock \textit{Probability Theory and Related Fields}, \textbf{77} 87--97.
\newblock Correction: {\bf 80}, 633 (1989).

\bibitem[{Stein and Shakarchi(2010)}]{SteiShak2010}
\textsc{Stein, E.~M.} and \textsc{Shakarchi, R.} (2010).
\newblock \textit{Complex analysis}, vol.~2.
\newblock Princeton University Press.

\bibitem[{Yosida(1980)}]{Yosi1980}
\textsc{Yosida, K.} (1980).
\newblock \textit{Functional analysis}, vol. 123 of \textit{Grundlehren der
  Mathematischen Wissenschaften [Fundamental Principles of Mathematical
  Sciences]}.
\newblock Sixth ed. Springer-Verlag, Berlin-New York.

\end{thebibliography}
\end{document}